\documentclass[11pt, english]{amsart}
\usepackage{fullpage}
\usepackage{amsmath}
\usepackage{amsfonts}
\usepackage{amssymb}
\usepackage{graphicx}

\usepackage{mathrsfs}
\usepackage{epsfig}

\usepackage{amsthm}

\usepackage{verbatim}
\usepackage{graphicx}
\usepackage{color}

\newcommand{\blue}{\color{blue}}

\newcommand{\black}{\color{black}}

\usepackage{url}

\newcommand{\beq}{\begin{equation}}
\newcommand{\eeq}{\end{equation}}

\newcommand{\e}{\mathrm{e}}
\newtheorem{theorem}{Theorem}[section]
\newtheorem{defi}[theorem]{Definition}
\newtheorem{lemma}[theorem]{Lemma}
\newtheorem{cor}[theorem]{Corollary}
\newtheorem{proposition}[theorem]{Proposition}
\newtheorem{rem}[theorem]{Remark}
\newcommand{\dd}{\,\mathrm{d}}
\newcommand{\ddo}{\mathrm{d}}

\author{Alexander Ostermann}
\address{Department of mathematics, University of Innsbruck,
Technikerstr.~13, 6020 Innsbruck, Austria (A. Ostermann)}
\email{alexander.ostermann@uibk.ac.at}

\author{Fr\'ed\'eric Rousset}
\address{Laboratoire de Math\'ematiques d'Orsay (UMR 8628), Universit\'e Paris-Sud,  91405 Orsay Cedex, France (F. Rousset)}
\email{frederic.rousset@math.u-psud.fr}

\author{Katharina Schratz}
\address{Fakult\"{a}t f\"{u}r Mathematik, Karlsruhe Institute of Technology,
Englerstr.~2, 76131 Karlsruhe, Germany (K. Schratz)}
\email{katharina.schratz@kit.edu}


\begin{document}

\begin{abstract}
We present a new filtered low-regularity Fourier integrator for the cubic nonlinear Schr\"odinger equation based on recent time discretization and filtering techniques. For this new scheme, we perform a rigorous error analysis and establish better convergence rates at low regularity than known for classical schemes in the literature so far. In our error estimates, we combine the better local error properties of the new scheme with a stability analysis based on general discrete Strichartz-type estimates. The latter allow us to handle a much rougher class of solutions as the error analysis can be carried out directly at the level of $L^2$ compared to classical results \black in dimension $d$, \black which are limited to higher-order (sufficiently smooth) Sobolev spaces $H^s$ with $s>d/2$.
In particular, we are able to establish a global error estimate in $L^2$ for $H^1$ solutions which is roughly of order $\tau^{ {1\over 2} + { 5-d \over 12} }$ in dimension $d \leq 3$ ($\tau$ denoting the time discretization parameter). This breaks the ``natural order barrier'' of $\tau^{1/2}$ for $H^1$ solutions which holds for classical numerical schemes (even in combination with suitable filter functions).
\end{abstract}

\title[Error estimates of a Fourier integrator for the cubic Schr\"odinger equation at low regularity]{Error estimates of a Fourier integrator for the cubic Schr\"odinger equation at low regularity}

\maketitle

\section{Introduction}
We consider the cubic nonlinear Schr\"odinger equation
\begin{equation}\label{nls}
\begin{aligned}
i \partial_t u &= - \Delta u + \vert u \vert^2 u, \qquad (t,x) \in \mathbb{R} \times \mathbb{R}^d
\end{aligned}
\end{equation}
in dimension $d \leq 3$. This equation, and more generally semi-linear Schr\"odinger equations,
\begin{equation}\label{nlsO}
i\partial_t u = \Delta u + \mu \vert u \vert^{2p}u, \qquad p \in \mathbb{N},\,\, \mu = \pm 1
\end{equation}
are  numerically well studied. To approximate the time evolution of \eqref{nlsO} various (time) discretization techniques have been proposed in the literature based on, e.g.,  splitting the right-hand side into the linear and nonlinear part (splitting schemes) or discretizing Duhamel's formula (exponential integrators), see, e.g.,  \cite{BeDe02,CanG15,CCO08,CoGa12,Duj09,Faou12,GauLu,Gau11,Ignat11,IZ09,Lubich08,Ta12} and the references therein.

For smooth solutions the error behaviour of these classical schemes is nowadays well understood and, based on a rigorous error analysis, global error estimates could be established. In the error estimates the regularity of the solution plays a crucial role and convergence (of a certain rate) only holds for sufficiently smooth solutions. One of the reasons for this regularity requirement is the following.
%
Within the construction of all (classical) numerical methods the stiff part (i.e., the term involving the differential operator $-\Delta$) is approximated in a way that the control of the local error requires the boundedness of additional spatial derivatives of the exact solution. Therefore, convergence of a certain order only holds under sufficient additional regularity assumptions on the solution. The severe order reduction of classical numerical schemes in case of non-smooth solutions is nowadays a well established fact in  numerical analysis, see, e.g., \cite{Eil16,Faou12,JL00,OS18} in case of (non)linear Schr\"odinger equations.

More precisely, classical schemes for \eqref{nlsO} \black with time step size $\tau$ \black introduce a local error that behaves roughly like (cf.~\cite{HochOst10,Lubich08,OS18})
\begin{equation}\label{localClass}
\tau^{1+\gamma} (-\Delta)^\gamma u(t),
\end{equation}
$u$ being the exact solution, such that convergence of order $\tau^\gamma$ in $H^s$ requires solutions in $H^{s+2\gamma}$.

Recently, a Fourier integrator for Schr\"odinger equations was introduced in \cite{OS18}. The main interesting property of the new scheme lies in the fact that the boundedness of only one additional derivative of the exact solution is required thanks to a local error structure of type
\begin{equation}\label{localNew}
\tau^{1+\gamma} \vert \nabla\vert^\gamma u(t)
\end{equation}
such that convergence of order $\tau^\gamma$ in $H^s$ requires solutions only in $H^{s+\gamma}$.

While the new discretization technique presented in \cite{OS18} allowed us to cut down the regularity assumption in the local error (cf.~\eqref{localClass} and \eqref{localNew}, respectively), the \emph{stability analysis in low regularity spaces} remained an open problem. This is due to the fact that the error analysis of low regularity integrators was up to now only based on classical tools. \black For estimating \black the nonlinear terms (in the global error) classical bilinear estimates based on Sobolev embedding are exploited. \black Note that this \black is a common approach in the error analysis of nonlinear dispersive equations, see, e.g., \cite{Eil16,Faou12} in case of semi-linear Schr\"odinger equations and \cite{BS19,BFS17,HS16,OS18} in the context of low regularity integrators. This classical approach easily allows us to prove stability of the numerical scheme at the cost that it requires highly regular solutions. More precisely, the analysis in \cite{OS18} is restricted \black in dimension $d$ \black to higher-order (sufficiently smooth) Sobolev spaces
\begin{align}\label{sob}
H^s \qquad \text{with Sobolev exponent} \quad s >d/2
\end{align}
for which $H^s$ is an algebra.
The latter assumption allows us  to establish the global error estimate
\begin{equation}\label{errBefore}
\Vert u(t_n) - u^n\Vert_s \leq c \tau^{\gamma} \quad \text{for solutions } u \in H^{s+\gamma} \quad \text{for } s >d/2.
\end{equation}
\black Here, $u^n$ denotes the numerical approximation to the exact solution $u(t)$ at time $t=t_n=n\tau$. \black While the condition $s > d/2$ is common in classical error analysis of nonlinear problems, it drastically increases the regularity assumptions on the solution: classical convergence estimates (such as \eqref{errBefore}) are restricted to the class of solutions in $H^{d/2+\varepsilon+\gamma}$ ($\varepsilon>0$) which is particularly limiting in higher dimensions $d \geq 2$.


While, from a numerical point of view, the analysis of nonlinear problems at low regularity is still (in large parts) widely open, the difficulty in the control of the nonlinear terms in low regularity spaces could be overcome in many cases at a continuous level. For the Schr\"odinger equation \eqref{nlsO} it is, for instance, a well-established fact (see, for example, the books \cite{Cazenave,Linares,Tao06}) that the Cauchy problem for \eqref{nlsO} on $\mathbb{R}^d$ is locally well-posed in $L^2$ for $2p \leq \frac4d$ and in $H^1$ for $2p \leq \frac4{d-2}$. The essential tool  in the well-posedness analysis in low regularity spaces are Strichartz estimates. In case of the free Schr\"odinger flow $S(t) = \e^{i t \Delta}$ on $\mathbb{R}^d$ they take the form
\begin{align}\label{str}
\left \Vert \e^{i t \Delta}u_0 \right\Vert_{L_t^q L_x^r} \leq c_{d,q,r} \Vert u_0\Vert_2 \qquad \text{for}\quad 2 \leq q, r \leq \infty, \quad \frac{2}{q}+\frac{d}{r} = \frac{d}{2}, \quad (q,r,d) \neq (2, \infty,2).
\end{align}
A natural question is what can we gain from them numerically. In particular, as for parabolic evolution equations, the so-called \emph{parabolic smoothing property}
\begin{align}\label{par}
\left \Vert \black (-\Delta)^{\alpha} \e^{ t \Delta} u_0 \black \right \Vert_{L_x^r} \leq c_{d,r} t^{- \alpha} \Vert u_0\Vert_{L_x^r}, \qquad \alpha \geq 0
\end{align}
is highly exploited in numerical analysis (see, e.g., the recent result \cite{para}). The main difficulty from a numerical point of view is that Strichartz estimates \eqref{str} are, in contrast to \eqref{par}, not pointwise in time and their gain lies in integrability and not differentiability. In particular, Strichartz-like estimates do not hold for the time (nor fully) discrete Schr\"odinger group $\{e^{in\tau \Delta}\}_{ n \in \mathbb{N}}$, see, e.g., the important works \cite{Ignat11,IZ09,IZ06,SK05}.

In \cite{Ignat11,IZ09} a new filtered splitting approximation was introduced for the nonlinear Schr\"odinger equation on $\mathbb{R}^d$, based on filtering the high frequencies in the linear part $S_\tau(t) \varphi = S(t) \Pi_{\tau^{- { 1 \over 2 } }} \varphi$ with the filter function
\[
\widehat{\Pi_{\tau^{- {1\over 2}}} \varphi}(\xi) = \hat{\varphi}(\xi) {\bf 1}_{\{ \vert \xi \vert \leq \tau^{-1/2}\}}, \qquad \xi \in \mathbb{R}^d.
\]
These filtered groups $S_\tau(t)$ admit discrete Strichartz-like estimates which are discrete in time and uniform in the time discretization parameter. The latter allows one to show stability of the scheme in the same space \black where the stability of the PDE is established. \black For these filtered schemes (of classical order one) error bounds of order one could be established in $L^2$ for solutions in $H^2$ for semilinear Schr\"odinger equations. \black In the preprint \cite{Choi}, this result was be extended to the semi discrete (time) analysis of the filtered Lie splitting scheme for $H^1$ solutions at the price of reduced order $\tau^\frac{1}{2}$ for time convergence -- the natural order barrier of classical numerical schemes at this level of regularity. \black

Let us also mention the paper \cite{Lubich08}, where the error of the second-order Strang splitting scheme for nonlinear Schr\"odinger and Schr\"odinger--Poisson equations was analysed. \black In this paper, Lubich's sophisticated argument allowed for the first time \black a rigorous second-order convergence bound of Strang splitting for the cubic nonlinear Schr\"odinger (NLS) equation in $L^2$  for exact solutions in $H^4$ (the natural space of regularity for classical second-order methods). The idea is to first prove fractional convergence  of the schemes in a suitable higher-order Sobolev space which implies a priori the boundedness of the numerical scheme in this space. This then allows one to establish error estimates in lower-order Sobolev spaces as  classical bilinear estimates can be applied in the stability argument with the numerical solution measured in a stronger norm. As the scaling of dimension and order of convergence play an important role, the argument does, however, not apply to solutions in
$$
H^s\quad \text{with}\quad s < d/2.
$$
Solutions with this regularity do not leave any room to play in the bootstrap argument.

In \black the present \black work, we introduce a new filtered low-regularity Fourier integrator based on the time discretization technique introduced in \cite{OS18} and inspired by the filtering of high frequencies \cite{Ignat11,IZ09}. The good properties of the new scheme together with a fine error analysis allow us to establish better convergence rates at low regularity than known in the literature so far, in particular, compared to our previous work \cite{OS18} on low-regularity integrators which was restricted to sufficiently smooth Sobolev spaces $H^s$ with $s>d/2$. With the aid of general discrete Strichartz-type estimates, we can overcome this limitation and prove $L^2$ estimates for the new scheme for solutions in $H^1$ in dimensions $d \leq 3$.

This approach in particular allows us to break the ``natural order barrier'' of $\tau^{1/2}$ for $H^1$ \black solutions. Note that the \black latter cannot be overcome by classical numerical schemes (not even by introducing suitable filter functions) due to their classical error structure of type $\tau^{\delta}(-\Delta)^{\delta} u$, introduced by the leading second order differential operator $-\Delta$.

\section{A Fourier integrator for the cubic Schr\"odinger equation at low regularity, the main theorem and \black the \black central idea of the proof}\label{sec:n}

In order to approximate the solution $u(t)$ of \eqref{nls} at time $t= t_{n+1} = t_n+\tau$ we choose the \black one-step method \black
\begin{equation}\label{scheme}
\begin{aligned}
u^{n+1} &=\Phi^\tau_K(u^n):={e}^{i \tau \Delta} \left( u^n - i \tau \Pi_{K} \left(\left(\Pi_K u^n\right)^2 \varphi_1(-2i \tau\Delta) \Pi_K \overline u^n\right)\right),\\
u^0 &= \Pi_{K} u(0)
\end{aligned}
\end{equation}
with $\varphi_1(z) = \frac{{e}^{z} - 1}{z}$ and the projection operator defined by the Fourier multiplier
\beq
\label{PiKdef}
\black \Pi_K =  \chi^2\left({ -i \nabla \over K}\right), \black
\eeq
which in Fourier space reads
$$
\widehat{\Pi_K\phi}(\xi) = \widehat{\phi}(\xi)\, \chi^2\left(\frac{\xi}{K}\right), \quad \xi \in \mathbb{R}^d.
$$
Here $\chi$ is a smooth radial nonnegative function which is one on $B(0,1)$ and supported in $B(0,2)$, and $K\geq 1$ is considered as \black a parameter that will depend on $\tau$. \black Note that, here, we will  not restrict ourselves to the choice $K= { \tau^{- {1 \over 2}}}$ as in \cite{Ignat11}, but we allow $K = { \tau^{- {\alpha \over 2}}}$ with \black some \black $\alpha \geq 1$. The main reason for this choice is that the introduction of the filter introduces a new term in the error. Indeed, by denoting by $u$ the exact solution of \eqref{nls} and by $u^n$ the sequence given by the scheme \eqref{scheme}, we have the estimate
\beq
\label{erreurintro}
\|u(t_{n}) - u^{n}\|_{L^2} \leq  \|u^n - u^K(t_{n})\|_{L^2} +  \|u^K(t_{n}) - u(t_{n})\|_{L^2},
\eeq
where $u^K(t)$ denotes the exact solution of the filtered PDE,
\beq
\label{nlsKintro}
i \partial_{t} u^K = - \Delta  u^K + \Pi_{K}(|\Pi_{K}u^K|^2 \Pi_{K}u^K),\qquad \black u^K(0) = \Pi_K u(0). \black
\eeq
We now observe that the scheme \eqref{scheme} is exactly the low-regularity Fourier integrator introduced in \cite{OS18}, \black applied to the filtered \black PDE \eqref{nlsKintro}. From this observation and due to the more favorable property of the local error of this scheme emphasized in \eqref{localNew}, we could expect an estimate of order $\tau$ for $\|u^n - u^K(t_{n})\|_{L^2}$ assuming only $H^1$ regularity of the exact solution. Nevertheless, for the second term on the right-hand side of \eqref{erreurintro}, i.e., $ \|u^K(t_{n}) - u(t_{n})\|_{L^2}$, we can get only an estimate of order $1/K$ for $H^1$ solutions. Therefore, the choice $K={ \tau^{- {1 \over 2}}}$, which yields uniform in $\tau$ discrete Strichartz-type estimates as proven in \cite{Ignat11}, would give a total error estimate of order $\tau^{1 \over 2}$, completely \black hiding the superior \black properties of the local error of the Fourier integrator. This is the reason for which we make the choice $K= { \tau^{- {\alpha \over 2}}}$ and choose $\alpha$ in the end in order to optimize the error. Taking $\alpha$ large makes the term $\|u^K(t_{n}) - u(t_{n})\|_{L^2}$ smaller, but the price to pay for such a choice is that there is a loss in the discrete Strichartz estimates. Indeed, we shall establish in Theorem \ref{theoDSE} below that the general form of the discrete Strichartz estimates reads
$$
\left\| e^{in \tau \Delta} \Pi_{K} f\right\|_{l^p_{\tau}L^q} \leq C  ( K \tau^{1\over 2 })^{2\over p }  \|f \|_{L^2}.
$$
\black (For the precise meaning of the norms, we refer to Section~\ref{sect:notations}.)\black We will also establish that this estimate with loss can be used to deduce an estimate with a uniform constant but with a loss of derivatives:
$$
\left\| e^{in \tau \Delta} \Pi_{K} f\right\|_{l^p_{\tau}L^q} \leq C  \|f \|_{H^{{2 \over p}( 1 - {1 \over \alpha})}}.
$$
Note that \black this type of loss of derivative \black in the Strichartz estimates also occurs in the case of compact manifolds \cite{Bour93a,Burq-Gerard-Tzvetkov}.

Choosing $\alpha$ larger than one will thus deteriorate the estimate that we get for the first term $\|u^n - u^K(t_{n})\|_{L^2}$.
In the end, by a careful choice of $\alpha$ such that the two terms contribute equally, we are able to get an estimate on the global error of the form
$$
\Vert u(t_n) - u^n\Vert_{L^2} \leq c \tau^{1/2+\gamma(d)} \quad \text{for solutions } u \in H^{1},
$$
where $\gamma(d)>0$ depends on the dimension $d$. Recall that such a favorable  error estimate  cannot  hold for classical numerical schemes (not even by introducing a suitable filter as done in the splitting schemes \cite{Ignat11,IZ09}) as for $H^1$ solutions the \black global error is proportional to \black $\tau^{1/2}$, in general, due to the local error structure \eqref{localClass}.

We conclude this section with the main theorem on the precise error estimates for our new scheme.
\begin{theorem}
\label{maintheo}
For every $T>0$ and $u_{0} \in H^1$, let us denote by $u\in \mathcal{C}([0,T], H^1)$ the exact solution of \eqref{nls} with initial datum $u_{0}$ and by $u^n$ the sequence defined by the scheme \eqref{scheme}. Then, there exist $\tau_{0}>0$ and $C_{T}>0$ such that for every step size $\tau \in (0, \tau_{0}]$, we have the following error estimates:
\begin{itemize}
\item if $d=1$, with the choice $K= 1/\tau^{5\over 6}$,
$$ \|u^n- u(t_{n})\|_{L^2} \leq C_{T} \tau^{5\over 6}, \quad 0 \leq n \leq N,$$
\item if $d=2$, with the choice $K= 1 /\tau^{3\over 4},$
$$ \|u^n- u(t_{n})\|_{L^2} \leq C_{T} \tau^{3\over 4}, \quad 0 \leq n \leq N,$$
\item if $d=3$, with the choice $K= 1 /\tau^{2\over 3},$
$$ \|u^n- u(t_{n})\|_{L^2} \leq C_{T} \tau^{2\over 3} \left|\log \tau\right|^{2 \over 3}, \quad 0 \leq n \leq N,$$
\end{itemize}
where $N$ is such that  $N\tau \leq T$.
\end{theorem}
In the above theorem we focused on $H^1$ solutions and optimized the rate of convergence. At the price of allowing a lower rate of convergence, we could handle even rougher data.
Note that we have analyzed only the defocusing equation \eqref{nls}. Nevertheless, the same results are true for the focusing one as long as the exact solution remains in $H^1$ (we recall that finite time blow-up in $H^1$ will occur in dimensions  $d=2, \,3$).

The \black rest of the \black paper is organized as follows.

In Section \ref{sectionstrichartz}, we describe the discrete Strichartz estimates, the proofs are postponed to Section~\ref{sectionproofstrichartz}. The aim of Section \ref{sectioncontinue} is to analyze the error $\|u^K(t_{n}) - u^n\|_{L^2}$. A crucial step towards the proof of Theorem \ref{maintheo} is performed in Section \ref{sectioncontdis}. Indeed, we prove that the exact solution $u^K$ of \eqref{nlsKintro} enjoys discrete Strichartz estimates, see Proposition \ref{propuKinfty}, that involve some loss of derivative or loss that is still better \black than that resulting from straightforward \black Sobolev embedding. These discrete Strichartz estimates for $u^K$ are needed for two reasons. At first, the structure of the local error described by \eqref{localNew} is a bit sketchy. A more precise description is given by (cf.~Corollary \ref{cor:locE})
$$
\tau^2 |\nabla u^K|^2 u^K(t)
$$
so that in order to control the local error in $L^2$ we need at least to control $\|\nabla u^K(t)\|_{L^4}$. Therefore, we need to rely on these discrete Strichartz estimates satisfied by the exact solution $u^K$ of the filtered PDE \eqref{nlsKintro} in order to estimate this part of the local error without using more regularity. The other part, where  the estimates of Proposition \ref{propuKinfty} are crucially used, is in the proof of the stability of the scheme at low regularity. Indeed, by defining $e^n= u^n- u^K(t_{N})$, we get that $e^n$ solves
$$
e^{n+1} =  e^{i \tau \Delta} \left( e^n - i \tau \Pi_{K} \left(  \varphi_{1}(-2i\tau \Delta) \Pi_{K} e^n  (\Pi_{K} u^K(t_{n})^2\right) \right) + \cdots
$$
where the dots stand for similar or quadratic and cubic terms with respect to $e^n$. Therefore, we get an $L^2$ estimate of the form
$$
\|e^{n+1} \|_{L^2} \leq \|e^n\|_{L^2} (1 + \tau \|u^K(t_{n})\|_{L^\infty}^2 + \cdots).
$$
In order to prove even boundedness of $e^n$, we need to prove that the expression
$$
\tau  \sum_{n=0}^N \|u^K(t_{n})\|_{L^\infty}^2
$$
is uniformly bounded with respect to $\tau$. This type of estimate will be a consequence of Proposition~\ref{propuKinfty}. Note that this uniform boundedness in dimension $d \geq 2$ cannot be obtained by using only the fact that $u^K \in \mathcal{C}([0, T], H^1)$.

In Sections \ref{sectionlocal1} and \ref{sectionlocal2}, we analyze the local error and finally, in Section \ref{sectionfinale}, we prove Theorem \ref{maintheo}.

\section{Notations}\label{sect:notations}

Note that the mild solution $u(t) = u(t,\cdot)$ of \eqref{nls} is given by
\begin{align}\label{duh}
u(t_n+\tau) = {e}^{i \tau \Delta} u(t_n) - i {e}^{i \tau \Delta} T(u)(\tau,t_n)
\end{align}
with the Duhamel operator
\begin{align}
T(u)(\tau,t_n) = \int_0^\tau {e}^{-i s \Delta} \vert u(t_n+s)\vert^2 u(t_n+s)\dd s.
\end{align}

\black Let $F$ be a function of two variables $(t,x)\in \mathbb{R}\times \mathbb{R}^d$. \black  We use the continuous norms
\begin{align*}
 &   \|F \|_{L^pL^q} = \left( \int_{\mathbb{R}} \|F(t, \cdot) \|_{L^q }^p \, \dd t \right)^{1 \over p}, \\
 &  \|F \|_{L^p_{T}L^q} = \left( \int_{0}^T \|F(t, \cdot) \|_{L^q }^p \, \dd t \right)^{1 \over p}
\end{align*}
with the convention that for $p= \infty$ the integral is replaced by \black the ess\,sup. \black

At the discrete level, for a sequence $(F_{k}(x))_{k \in \mathbb{Z}}$, we use the notation
$$
\| F\|_{l^p_{\tau}L^q} =   \| F_{k}\|_{l^p_{\tau}L^q}= \left( \tau  \sum_{k \in \mathbb{Z} } \|F_{k}\|_{L^q}^p \right)^{1\over p}
$$
and
$$
\| F_{k}\|_{l^p_{\tau, N}L^q} = \left( \tau  \sum_{k =0 }^N \|F_{k}\|_{L^q}^p \right)^{1\over p}.
$$
\black For $p=\infty$, $\tau$ times the sum is replaced by the supremum.

Finally, we write $a\lesssim b$ whenever there is a generic constant $C > 0$ such that $a \le C b$.
\black

\section{Continuous and discrete Strichartz estimates}
\label{sectionstrichartz}

Let us first recall the classical Strichartz estimates for the linear Schr\"{o}dinger equation.

Let us say that $(p,q)$ is admissible if $p \geq 2$, $q \geq 2$, $(p,q,d)\neq (2,\infty, 2)$ and ${2 \over p} +{ d \over q } = {d \over 2}.$  The admissible pair with $p=2$ is called the endpoint. Note that there is no such point in dimensions $1$ \black and $2$. \black
\black As usually, the dual indices of $(p,q)$ will be denoted by $(p',q')$, i.e., $\frac1p+\frac1{p'} = 1$ and $\frac1q+\frac1{q'} = 1$.\black

\begin{theorem}
For every $(p,q)$, admissible, there exists $C>0$ such that for every $f\in L^2$ and $F \in L^{p'} L^{q'}$
\begin{eqnarray}
\label{Tc}& &  \|e^{it \Delta} f \|_{L^p{L^q}} \leq C \| f\|_{L^2} \\
\label{T*c} & & \left \| \int_{\mathbb{R}} e^{-is \Delta } F(s,\cdot)\, ds\right\|_{L^2} \leq C \|F \|_{L^{p'} L^{q'}}.
\end{eqnarray}
Moreover, for every $(p_{1}, q_{1})$ and  $(p_{2}, q_{2})$ admissible, there exists $C>0$ such that for every $F \in L^{p_{2}'} L^{q'_{2}}$, we have
\begin{equation}
\label{TT*c} \left\| \int_{-\infty}^t e^{i ( t-s ) \Delta } F(s,\cdot) \, ds \right\|_{L^{p_{1}} L^{q_{1}}} \leq  C \|F \|_{L^{p_{2}'} L^{q'_{2}}}.
\end{equation}
\end{theorem}
These estimates were proven by Strichartz \cite{Strichartz} in a special case and by Ginibre and Velo \cite{Ginibre-Velo}. The endpoint $p=2$ for $d \geq 3$ was proven by Keel and Tao \cite{Keel-Tao}.

We shall  next study discrete versions of these inequalities for the group
\begin{equation}\label{SK}
S_{K} (t) = e^{i t \Delta } \Pi_{K} = \black \Pi_{K} e^{i t \Delta }. \black
\end{equation}
We will consider that $K \geq \tau^{-{1 \over 2}}$. In the case $K= \tau^{-{1 \over 2}}$ such estimates were established in \cite{IZ09}. This is  the only choice which ensures estimates without loss. Here, we will allow some loss depending on $K$ in order to optimize the total error.

\begin{theorem}
\label{theoDSE}
For every $(p,q)$ admissible with $p>2$, there exists $C>0$ such that for every $K$ and $\tau$ satisfying $K \tau^{1 \over 2} \geq 1$ \black and all $f\in L^2$, \black we have
\beq
\label{T}
\left\| S_{K}(n \tau) f\right\|_{l^p_{\tau}L^q} \leq C  ( K \tau^{1\over 2 })^{2\over p } \|f \|_{L^2}.
\eeq
For every $(p,q)$ admissible with $p>2$, there exists $C>0$ such that for every $K$ and $\tau$ satisfying $K \tau^{1 \over 2} \geq 1$ \black and all $F\in l^{p'}_\tau L^{q'}$, \black we have
\beq
\label{T*}
\left\|  \tau \sum_{n \in \mathbb{Z} }  S_{K} (-n\tau) F_{n} \right\|_{L^2} \leq C  ( K \tau^{1\over 2 })^{2\over p}  \| F\|_{l^{p'}_{\tau} L^{q'} }.
\eeq
For every $(p_{1},q_{1})$, $(p_{2}, q_{2}) $ admissible with $p_{1}>2,  \, p_{2}>2$, there exists $C>0$ such that for every $K$ and $\tau$ satisfying $K \tau^{1 \over 2} \geq 1$, \black all $s \in [-3,3]$ and all $F\in l^{p_2'}_\tau L^{q_2'}$, \black we have
\beq
\label{TT*}
\left\| \tau \black \sum_{k = -\infty}^{n-1} \black  S_{K} ((n-k+ s) \tau)  F_{k} \right\|_{l^{p_{1}}_{\tau}L^{q_{1}} }
\leq C ( K \tau^{1\over 2 } )^{ {2\over p_{1}} + { 2 \over p_{2} }   }  \| F\|_{l^{p_{2}'}_{\tau} L^{q_{2}'} }.
\eeq
\end{theorem}

Note that we have excluded the endpoints in the statements of the Strichartz estimates \eqref{T}, \eqref{T*} and \eqref{TT*}.
Also note that in the estimate \eqref{TT*}, we have added in the definition of the operator a shift $s \tau$. Though it almost  does not change anything in the proof, taking into account this shift will be crucial to get the estimates of Proposition \ref{uKgrid} and the control of the local error. The proof of Theorem \ref{theoDSE} is postponed to Section \ref{proofDSE}.

It will be useful to convert the estimates of Theorem \ref{theoDSE} when $K= \tau^{-{\alpha \over 2}}$ with $\alpha\geq 1$ (a choice that we will make in order to optimize the error estimate)  into estimates with uniformly bounded constants but with loss of derivatives.

\begin{cor}
\label{corDSE}
For every $(p,q)$ admissible with $p>2$, there exists $C>0$ such that for every $0<\tau \leq 1$ and $K = \tau^{-{\alpha\over 2}}$, $\alpha \geq 1$,  we have
\beq
\label{Tloss}
\left\| S_{K}(n \tau) f\right\|_{l^p_{\tau}L^q} \leq C    \|f \|_{H^{ {2 \over p}( 1 - {1 \over \alpha} ) }}\qquad \black \text{for all \ $f\in H^{ {2 \over p}( 1 - {1 \over \alpha} ) }$}. \black
\eeq
For every $(p, q)$ admissible with  $p>2$, there exists $C>0$ such that for every  $0<\tau \leq 1$, $K = \tau^{-{\alpha \over 2}}$, \black $\alpha \ge 1$ \black and $s \in [-8,8]$ we have
\beq
\label{TT*loss}
\left\| \tau \black \sum_{k=-\infty}^{n-1} \black  S_{K} ((n-k+ s) \tau)  F_{k} \right\|_{l^{p}_{\tau}L^{q} } \leq C   \| F\|_{l^1_\tau  H^{ {2 \over p}( 1 - {1 \over \alpha} ) } }\qquad \black
\text{for all \ $F\in l^1_\tau  H^{ {2 \over p}( 1 - {1 \over \alpha} ) }$}. \black
\eeq
\end{cor}
Note that, since $\tau^{1 \over p}\| \Pi_{K} f\|_{L^q} \leq  \left\| S_{K}(n \tau) f\right\|_{l^p_{\tau}L^q}$ the estimate \eqref{Tloss} also encodes the modified Sobolev estimate
\beq
\label{sobmod}
\tau^{1 \over p} \| \Pi_{K} f\|_{L^q} \leq C \|f \|_{H^{ {2 \over p}( 1 - {1 \over \alpha} ) }}.
\eeq
The proof of this estimate is postponed to Section \ref{sec:corDSE}.

\section{$H^1$ Cauchy problem for \eqref{nls}}
\label{sectioncontinue}

Let us recall the following well-known result for \eqref{nls}. We refer, for example, to the book \cite{Linares}.

\begin{theorem}
\label{theoNLS}
For $d \leq 3$ and for every $u_{0}\in H^1$, there exists for every $T>0$ a unique solution of \eqref{nls} in $\mathcal{C}([0,T],  H^1)$ such that $u(0)= u_{0}$. Moreover, this solution is such that  $u, \,\nabla u \in L^p_{T}L^q$ for every admissible $(p,q)$.
\end{theorem}

Note that in the focusing case, there exists under the same assumptions  a maximal $H^1$ solution defined on $[0, T^*)$ (and in this case $T^*$ can be finite) with similar properties. All our convergence estimates thus extend to the focusing case on $[0, T]$ for every $T<T^*$.

Let us now consider a frequency truncated equation
\beq
\label{NLSK}
i \partial_{t} u^K = - \Delta  u^K + \Pi_{K}(|\Pi_{K}u^K|^2 \Pi_{K}u^K), \quad  u^K(0)= \Pi_{K} u_{0}.
\eeq
As in Theorem \ref{theoNLS}, we can  easily get:

\begin{proposition}
\label{propNLSK}
For $d \leq 3$, $u_{0} \in H^1$, and $K \geq 1$, there exists a unique solution of \eqref{NLSK} such that $u^K \in \mathcal{C}([0,T],  H^1)$ for every $T \geq 0$. Moreover   $u^K, \,\nabla u^K \in L^p_{T}L^q$ for every admissible $(p,q)$. More precisely, for every $T \geq 0$ and every $(p,q)$ admissible, there exists $C_{T}>0$ such that for all $K\ge 1$ we have
$$
\| u^K\|_{L^p_{T}W^{1, q}} \leq C_{T}.
$$
\end{proposition}

We shall not detail the proof of this proposition that follows exactly the lines of the proof of Theorem \ref{theoNLS}.

\begin{rem}
\label{remlocal}
Note that, since  $\Pi_{2K} \Pi_{K}= \Pi_{K}$, we have that $\Pi_{2K} u_{K}$ solves the same equation \eqref{NLSK} with the same initial data and hence we have by uniqueness that
$$
\Pi_{2K} u^K (t) = u^K(t) \quad \text{for all \ $t \in [0,T]$.}
$$
\end{rem}

We can also easily get the following corollary.
\begin{cor}
\label{corNLSK}
For $d \leq 3$, $u_{0} \in H^1$ and every $T>0$, there exists $C_{T}>0$ such that for every $K \geq 1$, we have the estimate
$$
\|u -u^K \|_{L^\infty_{T}L^2} \leq { C_{T} \over K}.
$$
\end{cor}

This will allow us to discretize in time the projected equation for $u^K$ only.

\begin{proof}
Let us first take $M_{T}$ such that by using Theorem \ref{theoNLS} and Proposition \ref{propNLSK}, we have
\beq
\label{borne1}
\| u^K\|_{L^p_{T}W^{1, q}} + \| u\|_{L^p_{T}W^{1, q}}\leq M_{T}, \quad K \geq 1
\eeq
for  some $(p, q)$ admissible with $q$ such that $d<q<  2 + { 4 \over d-2}$ ($q<\infty$ if $d= 1, \, 2$) so that $W^{1,q}$ is embedded in $L^\infty$. Note that $M_{T}$ in general  depends on $T$. In the following and \black more generally, \black  we will denote by $M_{T}$ \black a generic constant \black that depends on $T$.

We \black further note that \eqref{borne1} in particular yields, by using successively the Sobolev embedding and H\"older's inequality, \black that
\beq
\label{borne2}
\|u^K\|_{L^2_{T}L^\infty} +\|u\|_{L^2_{T}L^\infty}  \lesssim  T^{ 1 - {2 \over p} }( \|u^K\|_{L^p_{T} W^{1, q}}  +  \|u\|_{L^p_{T} W^{1, q}}) \lesssim \black T^{ 1 - {2 \over p} } \black M_{T},
\eeq
where $(p,q)$ is admissible and $q$ is such that $d<q<  2 + { 4 \over d-2}$ ($q<\infty$ if $d= 1, \, 2$).

We first observe that
$$
\| u - \Pi_{K} u \|_{L^\infty_{T}L^2} \lesssim  {1 \over K} \| \nabla u \|_{L^\infty_{T}L^2} \lesssim  {M_{T} \over K}.
$$
By using Duhamel's formula, we have that
\begin{multline*}
u (t) - u^K(t) = e^{it \Delta}  (1-\Pi_{K}) u_{0}  - i \int_{0}^t e^{i (t-s ) \Delta }  \Pi_{K}( |u|^2 u - |\Pi_{K}u|^2 \Pi_{K} u  ) \dd s  \\
- i \int_{0}^t e^{i (t-s ) \Delta } \Pi_{K} ( |\Pi_{K}u|^2 \Pi_{K} u - |\Pi_{K}u^K|^2 \Pi_{K} u^K  ) \dd s - i \int_{0}^t e^{i (t-s ) \Delta } (1- \Pi_{K}) ( |u|^2 u )\dd s.
\end{multline*}
From standard  estimates, we then obtain that for every $T_{1} \leq T$,
\begin{multline*}
\|u- u^K \|_{L^\infty_{T_{1}}L^2} \leq {C \over K } +  C \|u- \Pi_{K} u \|_{L^\infty_{T}L^2} \Bigl( \|u\|_{L^2_{T}L^\infty}^2 +   \|\Pi_{K}u\|_{L^2_{T}L^\infty}^2\Bigr) \\
+  C \| u - u^K \|_{L^\infty_{T_{1}}L^2} \Bigl( \|u\|_{L^2_{T}L^\infty}^2 + \|u^K\|_{L^2_{T}L^\infty}^2\Bigr) +  {C \over K} \| u \|_{L^\infty_{T}H^1} \|u\|_{L^2_{T}L^\infty}^2,
\end{multline*}
where $C>0$ is a number independent of $T$ and $T_{1}$. Consequently, by using \eqref{borne1} and \eqref{borne2}, we obtain that
$$
\|u- u^K \|_{L^\infty_{T_{1}}L^2} \leq {M_{T}  \over K} + C T_{1}^{ 2 - {4 \over p} }  \|u- u^K \|_{L^\infty_{T_{1}}L^2} M_{T}^2,
$$
where $p$ is in particular such that $2/p <1$. Consequently, we can choose $T_{1}$ sufficiently small such that
$$
\black C T_{1}^{ 2 - {4 \over p} }  M_{T}^2 \leq {1 \over 2} \black
$$%
and we obtain that
$$
\|u- u^K \|_{L^\infty_{T_{1}}L^2} \leq { 2M_{T}  \over K} .
$$
This proves the desired estimate on $[0, T_{1}]$. We can then perform the same argument on $[T_{1}, 2T_{1}], \cdots$ to finally get that
$$
\|u- u^K \|_{L^\infty_{T}L^2} \leq { C _{T} \over K},
$$
where $C_{T}$ behaves like \black $e^{CT}M_T$. \black
\end{proof}

\section{Discrete Strichartz estimates of the exact solution}
\label{sectioncontdis}

In this section, we shall  prove that the sequence $(u^K(t_{k}))_{0 \leq k \leq N}$ where $u^K$ solves \eqref{NLSK} satisfies  discrete Strichartz estimates. This will be important in the following to estimate the local error and to control the stability of the scheme.

Let us first notice that by the Sobolev embedding $H^1 \subset L^q$, we have thanks to Proposition \ref{propNLSK} an estimate $\| u^K(t_{k}) \|_{l^p_{\tau, N}L^q} \leq  C_{T}$ for every $(p,q)$ admissible. Nevertheless, this is not sufficient for our purpose. Indeed, for the estimate of the local error, we shall also need discrete  Strichartz estimates of $\nabla u^K(t_{k})$. Moreover, to prove the stability of the scheme, we shall also need an estimate without loss of the form  $ \| u^K(t_{k}) \|_{l^2_{\tau, N}L^\infty} \leq  C_{T} $ that does not follow from Sobolev embedding in dimensions~$2$ and $3$.

Let us start with  an estimate that will ensure a uniform control of $\| u^K(t_{k}) \|_{l^2_{\tau, N}L^\infty}$. This will be crucial in the proof of the stability of the scheme.

\begin{defi}
Let $K= \tau^{-{\alpha \over 2}}$ \black for some \black $\alpha \geq 1$. We say that $(p,q, \sigma)$ verifies property {\bf (H)} if:
$$
(p,q) \mbox{ is admissible}, \quad p>2, \quad \sigma q>d, \quad  \sigma+ {2 \over p} \left(1 - {1 \over \alpha}\right) \leq 1.
$$
\end{defi}

\begin{rem}
\label{remcond}
Let us  check \black that the set of triples \black $(p,q,\sigma)$ verifying {\bf (H)} is not empty. In dimension~$1$, we can clearly take $q=2$, $p=\infty$ and any $\sigma \in (1/2, 1]$ due to Sobolev embedding.

In dimension $2$, by taking $\sigma = {2 \over q} + \epsilon$, $\epsilon >0$, (note that it is enough to have the embedding $W^{\sigma, q}\subset L^\infty$), $(p,q,\sigma)$ verifies {\bf (H)} if  $(p,q)$ is admissible, $p>2$ and $ \epsilon - { 2 \over p\alpha} \leq 0$. It can be satisfied for any $(p,q)$ admissible with $p<\infty$ by taking $\epsilon = { 1 \over p \alpha}$.

In dimension $3$, the set of $(p,q,\sigma)$ verifying {\bf (H)} is not empty if $\alpha <2$. Indeed, by  taking again $\sigma = {3 \over q} + \epsilon$, $\epsilon >0$,  we need to verify ${1 \over 2 }+ \epsilon \leq  {2 \over p \alpha}$ which means that we can find $\epsilon$ if $ 1 \leq \alpha <{4\over p}.$  Consequently, if $\alpha <2$, we can find $p>2$ such that this is satisfied.
\end{rem}

\begin{proposition}
\label{propuKinfty}
\black Let $K=\tau^{-{\alpha \over 2}}$ for some $\alpha \geq 1$, \black $\alpha <2$ in dimension $3$. \black Further, let $(p,q,\sigma)$ verify \black property {\bf (H)}. Then for every $T>0$, there exists $C_{T}$ such that for every $\tau \in (0, 1]$ and every $\hat s \in [- 2\tau ,  2\tau], $ we have the estimate
\beq
\label{Linftypq}
\sup_{ s \in [0, \tau]}\|e^{i \hat s  \Delta}u^K(t_{k}+s)\|_{l^p_{\tau, \lfloor N-{s \over \tau}\rfloor}W^{\sigma, q}} \leq C_{T}.
\eeq
\end{proposition}

The crucial consequence of \black this proposition \black is that, under the above assumptions and in the particular case when $\hat s=0$, we get by Sobolev embedding that
\beq
\label{uKinfty}
\sup_{s \in [0, \tau]} \|u^K(t_{k}+s)\|_{l^2_{\tau, \lfloor N-{ s\over \tau}\rfloor}L^\infty} \leq C_{T}.
\eeq
\black In particular, this implies that \black
$$
\|u^K(t_{k})\|_{l^p_{\tau, N}L^\infty} \leq C_{T}
$$
for $p>2$ such that $(p,q,\sigma)$ verifies {\bf (H)}. In dimensions $1$ and $2$, there is no restriction on $\alpha$. In dimension $3$, this only requires that $\alpha <2$.

Note that another useful consequence of \eqref{Linftypq} is that, though $\varphi_{1}(i \tau \Delta)$ is not continuous on $L^q$ for $q \neq 2$ with uniform estimate with respect to $\tau,$ we have \black the following bound. \black

\begin{cor}
\label{corfilter}
For $(p,q,\sigma)$ verifying {\bf (H)}, we also obtain that
\beq
\label{filterestimate}
\|\varphi_{1}(2i \tau \Delta) u^K(t_{k})\|_{l^p_{\tau, N}W^{\sigma, q}} \leq C_{T}.
\eeq
\end{cor}

We will start with the proof of Proposition \ref{propuKinfty}.

\begin{proof}[Proof of Proposition \ref{propuKinfty}]
We first prove the estimate \eqref{Linftypq} for $\hat s=0$.

We use Duhamel's formula to get that for every $0 \leq n \leq N$ and $s \in [0, \tau]$,
$$
u^K(t_{n}+s)= e^{i (t_{n} +s)  \Delta} \Pi_{K}u_{0}  - i   \int_{0}^{t_{n}+s} e^{i(t_{n}+s- s_{1})\Delta} \Pi_{K} (|\Pi_{K}u^K|^2 \Pi_{K}u^K)(s_{1}) \dd s_{1}
$$
that we rewrite as
\begin{multline*}
u^K(t_{n}+s)= S_{K}(t_n+s)u_{0} - i S_{K}(t_n+s) \sum_{k=0}^{n-1} \int_{0}^\tau e^{-i (t_k+\tilde s) \Delta }  (|\Pi_{K}u^K|^2 \Pi_{K}u^K)(t_k+\tilde s) \dd \tilde s \\
- i \int_{0}^s \black e^{-i (\tilde s - s) \Delta } \Pi_{K} \black (|\Pi_{K}u^K|^2 \Pi_{K}u^K)(t_n + \tilde s)\dd \tilde s.
\end{multline*}
Therefore,
\begin{multline}
\label{expand0}
u^K(t_{n}+s)= S_{K}(t_n) e^{is \Delta}u_{0} - i  \sum_{k=0}^{n-1} \int_{0}^\tau S_{K}(t_{n-k} + s - \tilde s ) (|\Pi_{K}u^K|^2 \Pi_{K}u^K)(t_k + \tilde s)\dd \tilde s\\
- i   \int_{0}^s  S_{K}(s - \tilde s )(|\Pi_{K}u^K|^2 \Pi_{K}u^K)(t_n + \tilde s)\dd \tilde s.
\end{multline}

Let us fix \black $M_{T}$ \black such that
\beq
\label{hypdepart}
\|u^K\|_{L^\infty_{T}H^1} + \|u^K\|_{L^p_{T}W^{\sigma, q}} \leq M_{T}
\eeq
and define $N_{1}$, $T_{1}=N_{1} \tau \leq T- \tau$. We shall first prove that we can find $T_{1}$ sufficiently small depending only on $M_{T}$ such that
\beq
\label{hyperc}
\sup_{s \in [0, \tau]} \|u^K(t_{k}+s)\|_{l^p_{\tau, N_{1}}W^{\sigma, q}} \leq  R M_{T}
\eeq
for some $R>0$ well-chosen.

Let us first observe that  by elliptic regularity, for $q\in (1,\infty)$, we have
$$ \|u^K(t_{k}+s)\|_{l^p_{\tau, 0}W^{\sigma, q}} \lesssim   \|(I- \Delta)^{\sigma \over 2} u^K(t_{k}+s)\|_{l^p_{\tau, 0}L^{q}},$$
therefore, since $\Pi_{2K}u^K= u^K$, 
we can use the  modified Sobolev estimate \eqref{sobmod} to get \black 
\beq
\label{unpas} \|u^K(t_{k}+s)\|_{l^p_{\tau, 0}W^{\sigma, q}} \lesssim   \|(I- \Delta)^{\sigma \over 2} u^K(t_{k}+s)\|_{l^p_{\tau, 0}L^{q}}
 \leq \black C_{0} \|u^K\|_{L^\infty_\tau H^{\sigma_1}} \black \leq C_{0} M_{T},
\eeq
where $\sigma_{1}= \sigma+ {2 \over p} (1 - {1 \over \alpha}) \leq 1$. We shall thus take $R=2C_{0}$ in \eqref{hyperc}. Next, for $n+1 \leq N_{1}$, assuming that $ \sup_{s \in [0, \tau]} \|u^K(t_{k}+s)\|_{l^p_{\tau, n}W^{\sigma, q}} \leq  2C_{0} M_{T}$, we get  by using \eqref{expand0} and the Strichartz estimate of Corollary \ref{corDSE} that
\begin{multline}
\label{uKgridinfty1}
\|u^K(t_{k}+s)\|_{l^p_{\tau, n+1}W^{\sigma, q}} \leq   C_{0} \|u_{0}\|_{H^1} + {C \over \tau} \int_{0}^\tau \| |\Pi_{K}u^K|^2 \Pi_{K} u^K (t_k + \tilde s) \|_{l^1_{\tau,n} H^{\sigma_{1}}} \dd \tilde s \\
+ { C \over \tau} \int_{0}^\tau \| |\Pi_{K}u^K|^2 \Pi_{K} u^K (t_k + \tilde s) \|_{l^1_{\tau,n+1} H^{\sigma_{1}}} \dd \tilde s.
\end{multline}
Next, we can use that
$$
\||\Pi_{K}u^K|^2 \Pi_{K} u^K (t_k + \tilde s) \|_{l^1_{\tau,n} H^{\sigma_{1}}} \leq  \| |\Pi_{K}u^K|^2 \Pi_{K} u^K (t_k + \tilde s) \|_{l^1_{\tau,n} H^{1}}
\leq \|u^{K}\|_{L^\infty_{T}H^1}  \|u^K(t_k + \tilde s)\|^2_{l^2_{\tau, n}L^\infty}.
$$
Since by the Sobolev and H\"older inequalities, we have
$$
\|u^K(t_k + \tilde s)\|^2_{l^2_{\tau, n}L^\infty}\lesssim  T_{1}^{ 1- {2 \over p}}  \black \|u^K(t_{k}+\tilde s)\|_{l^p_{\tau, n}W^{\sigma, q}}^2, \black
$$%
we get  from the induction assumption that
$$
\||\Pi_{K}u^K|^2 \Pi_{K} u^K (\tilde s + t_{k}) \|_{l^1_{\tau,n} H^{\sigma_{1}}} \leq T_{1}^{ 1- {2 \over p}} M_{T} (2C_{0} M_{T})^2.
$$
In a similar way, we also obtain that
$$
\| |\Pi_{K}u^K|^2 \Pi_{K} u^K (t_k + \tilde s) \|_{l^1_{\tau,n+1} H^{\sigma_{1}}} \leq T_{1}^{ 1- {2 \over p}} M_T \black \|u^K(t_k + \tilde s)\|_{l^p_{\tau, n+1}W^{\sigma, q}}^2, \black
$$%
and we use that
$$
\|u^K(t_{k}+s)\|_{l^p_{\tau, n+1}W^{\sigma, q}} \leq   \|u^K(t_{k}+s)\|_{l^p_{\tau, n}W^{\sigma, q}} + \tau^{1 \over p}  \|u^K(t_{n+1}+s)\|_{W^{\sigma, q}},
$$
which gives from the modified Sobolev embedding \eqref{sobmod}
$$
\|u^K(t_{k}+s)\|_{l^p_{\tau, n+1}W^{\sigma, q}} \leq   \|u^K(t_{k}+s)\|_{l^p_{\tau, n}W^{\sigma, q}} +  M_{T}.
$$
Consequently, by plugging these estimates into \eqref{uKgridinfty1}, we obtain that
$$
\|u^K(t_{k}+s)\|_{l^p_{\tau, n+1}W^{\sigma, q}} \leq   \black C_{0} M_{T}+ C T_{1}^{1 - {1 \over p}} M_{T}^3 + 8 C T_{1}^{1 - {1 \over p}}C_{0}^2 M_{T}^3. \black
$$%
This yields
$$
\sup_{s \in [0,\tau]}\|u^K(t_{k}+s)\|_{l^p_{\tau, n+1}W^{\sigma, q}} \leq  2  C_{0} M_{T}
$$
by choosing $T_{1}$ sufficiently small (note that $T_{1}$ depends only on $M_{T})$. This allows one to get by induction that
$$
\sup_{s \in [0, \tau]}  \|u^K(t_{k}+s)\|_{l^p_{\tau, N_{1}}W^{\sigma, q}} \leq   2C_{0} M_{T}.
$$
Since $T_{1}$ only depends on $M_{T}$, we can iterate the argument on  $[T_{1}, 2T_{1}]$, ... to finally get 
$$ \sup_{s \in [0, \tau]}  \|u^K(t_{k}+s)\|_{l^p_{\tau, N-1}W^{\sigma, q}} \leq   C_{T}.$$
Note that this also yields
 $$  \|u^K(t_{k})\|_{l^p_{\tau, N}W^{\sigma, q}} \leq   C_{T}.$$
  Indeed, we have that 
  $$  \|u^K(t_{k})\|_{l^p_{\tau, N}W^{\sigma, q}} \lesssim    \|u^K(t_{k})\|_{l^p_{\tau, N-1}W^{\sigma, q}} +  \tau^{ 1 \over p}
   \|u^K(t_{n}) \|_{W^{\sigma, q}} \leq C_{T}$$
  since  by using the same estimates as in  \eqref{unpas}, we have
$$  \tau^{ 1 \over p} \|u^K(t_{n})\|_{W^{\sigma, q}} \lesssim 
  \black  \|u^K\|_{L^\infty_\tau H^{\sigma_1}} \black \leq M_{T}.$$
  This proves \eqref{Linftypq} in the case $\hat s = 0.$

To get the estimate in the general case, we apply  $e^{i \hat s \Delta}$ to \eqref{expand0} to get
\begin{multline*}
e^{i \hat s \Delta} u^K(t_{n}+s)= S_{K}(t_n) e^{i(s+ \hat s)  \Delta}u_{0} - i  \sum_{k=0}^{n-1} \int_{0}^\tau  S_{K}(t_{n-k} +s + \hat s  - \tilde s ) (|\Pi_{K}u^K|^2 \Pi_{K}u^K)(t_k + \tilde s )\dd \tilde s \\
- i \int_{0}^s  S_{K}(s + \hat s- \tilde s )(|\Pi_{K}u^K|^2 \Pi_{K}u^K)(t_n + \tilde s )\dd \tilde s.
\end{multline*}
From the same use of the Strichartz estimates of Corollary \ref{corDSE}  as above, we obtain that
$$
\|e^{i \hat s \Delta }u^K(t_{k}+s)\|_{l^p_{\tau, N-1}W^{\sigma, q}} \leq C_{0}\|u_{0}\|_{H^1} +  M_{T} T^{1 - {1 \over p}} \sup_{s \in [0, \tau]}  \|u^K(t_{k}+s)\|_{l^p_{\tau, N-1}W^{\sigma, q}}^2.
$$
Since we have already proved the estimate \eqref{Linftypq} for $\hat s = 0$, this proves the estimate in the general case.
 Note that we can use the same trick as above to get the estimate for 
$  \|e^{i \hat s \Delta }u^K(t_{k})\|_{l^p_{\tau, N}W^{\sigma, q}}$.
\end{proof}

It remains to prove \eqref{filterestimate}.

\begin{proof}[Proof of Corollary \ref{corfilter}]
We first note that we can decompose
$$
\varphi_{1}(2i \tau \Delta) u^K(t_{k}) =   \varphi_{1}(2i \tau \Delta) ( 1 - \Pi_{\tau^{-{1 \over 2}}})u^K(t_{k}) + \varphi_{1}(2i \tau \Delta)  \Pi_{\tau^{-{1 \over 2}}}u^K(t_{k}).
$$
\black By using Lemma \ref{lemfilter1}, we have that  the multiplier $ \varphi_{1}(2i \tau \Delta)  \Pi_{\tau^{-{1 \over 2}}} $ is continuous on $L^q$ for every $q$ with norm uniform in $\tau$. 
Therefore, we get from Proposition \ref{propuKinfty} that
$$
\| \varphi_{1}(2i \tau \Delta) { \Pi_{\tau^{-{1 \over 2}}}\black} u^K(t_{k})\|_{l^p_{\tau, N}W^{\sigma, q}} \leq C \|u^K(t_{k})\|_{l^p_{\tau, N}W^{\sigma, q}} \leq C_{T}.
$$
To estimate the remaining part, we just observe that
$$
\varphi_{1}(2i \tau \Delta) ( 1 - \Pi_{\tau^{-{1 \over 2}}})u^K(t_{k}) = {1 - \Pi_{\tau^{-{1 \over 2}}} \over 2 i \tau \Delta }   e^{2 i\tau \Delta} u^K(t_{k})  -  {1 - \Pi_{\tau^{-{1 \over 2}}} \over 2 i \tau \Delta } u^K(t_{k}).
$$
\black Again, the multiplier  ${1 - \Pi_{\tau^{-{1 \over 2}}} \over 2 i \tau \Delta } $ is continuous on $L^q$ for every $q$ with norm uniform in $\tau.$, see \eqref{estfilter2} in Lemma \ref{lemfilter1}.\black Therefore, we obtain that
$$
\|\varphi_{1}(2i \tau \Delta) ( 1 - \Pi_{\tau^{-{1 \over 2}}})u^K(t_{k})\|_{l^p_{\tau, N}W^{\sigma, q}}\leq C\left(  \|e^{2 i\tau \Delta} u^K(t_{k})\|_{l^p_{\tau, N}W^{\sigma, q}} + \|u^K(t_{k})\|_{l^p_{\tau, N}W^{\sigma, q}}   \right)
$$
and the result follows by using again Proposition \ref{propuKinfty}.
\end{proof}

\begin{proposition}
\label{uKgrid}
For every $T \geq 0$,  $u_{0}\in H^1$ and for every $(p,q)$ admissible with  $p>2$, there exists $C_{T}>0$  such that for every $K$, $\tau$ as in Proposition \ref{propuKinfty}, with $\alpha <2$  in dimension 3, we have uniformly in \black $s \in [-2\tau, 2\tau]$ \black the estimate
\begin{align}
\label{nablauKgrid}
& \| e^{is \Delta} \nabla u^K(t_{k}) \|_{l^p_{\tau, N}L^q} \leq C_{T}  (K \tau^{1 \over 2})^{ 2 \over p}.
\end{align}
\end{proposition}

Note that the above proposition gives in particular an estimate of $\| \nabla u^K(t_{k})\|_{l^p_{\tau, N}L^q}$ in the special case $s=0$.

\begin{proof}
%
By using again \eqref{expand0}, we write
\begin{multline}
\label{expand1}
e^{i s \Delta } \nabla u_{K}(t_{n}) =S_{K}(n \tau)(e^{i s \Delta } \nabla u_{0})\\
- i  \int_{0}^\tau \sum_{k=0}^{n-1} S_{K}(t_{n-k} +s - \tilde s) \nabla \bigl(|\Pi_{K}u^K|^2 \Pi_{K}u^K\bigr)(t_k+\tilde s)\dd \tilde s.\qquad
\end{multline}
We can then  use Theorem \ref{theoDSE} (note that $(s- \tilde s)/\tau$ is uniformly bounded in \black $[-3, 3]$) \black to get
\begin{multline*}
\|e^{i s \Delta } \nabla u_{K}(t_{n})\|_{l^p_{\tau, N}, L^q} \leq C \Bigl( (K \tau^{1 \over 2})^{2 \over p} \|e^{i s \Delta } \nabla u_{0}\|_{L^2} \Bigr. \\
\Bigl.+ (K \tau^{1 \over 2})^{2 \over p} {1 \over \tau}  \  \int_{0}^\tau \sup_{\tilde s \in [0,\tau]} \left\| \nabla \bigl(|\Pi_{K}u^K|^2 \Pi_{K}u^K\bigr)(t_k+\tilde s) \right\|_{l^1_{\tau, N} L^2} \dd\tilde s\Bigr).
\end{multline*}
To estimate the last term in the above estimate, we use that \black
\begin{align*}
\sup_{\tilde{s} \in [0,\tau]} &\left\| \nabla \left(|\Pi_{K}u^K|^2 \Pi_{K}u^K\right)(t_{k}+ \tilde s) \right\|_{l^1_{\tau, N} L^2 }  \\
& \le \sup_{\tilde{s} \in [0,\tau]}  \left \| \|\nabla u^K(t_{k}+ \tilde s) \|_{L^2} \|\Pi_{K}u^K(t_{k}+ \tilde s) \|_{L^\infty}^2 \right\|_{l^1_{\tau, N}} \\
& \leq  \|\nabla u^K \|_{L^\infty_{T}L^2}  \sup_{\tilde s \in [0, \tau]}\|\Pi_{K} u^K(t_{k}+ \tilde s)\|_{l^2_{\tau, N}L^\infty}^2.
\end{align*} \black
To conclude, we can use the estimate \eqref{uKinfty} which holds even in dimension $3$ with the assumption that $\alpha <2$ by using Remark \ref{remcond} and Proposition \ref{propuKinfty}.
\end{proof}

\section{Local error analysis}
\label{sectionlocal1}

We shall now study the time discretization \eqref{scheme} of \eqref{NLSK}. By using Duhamel's formula, we get that
\begin{equation}\label{duh1}
u^K(t_n+\tau) = e^{i \tau \Delta} u^K(t_n) - i e^{i \tau \Delta} \black \Pi_{K}T(\Pi_k u^K)(\tau, t_n), \black
\end{equation}
where
\begin{align}
T(\Pi_K u^K)(\tau,t_n) =  \int_0^\tau e^{-i s \Delta} \Big[\vert \Pi_K u^K (t_n+s)\vert^2 \Pi_K u^K(t_n+s)\Big] \ddo s.
\end{align}
Iterating Duhamel's formula \eqref{duh1}, i.e., plugging the expansion
$$
u^K(t_n+s) =  e^{i s \Delta} u^K(t_n) - i e^{i s \Delta} \Pi_{K} T(\Pi_k u^K)(s,t_n)
$$
(which follows by replacing $\tau$ with $s$ in \eqref{duh1}) into \eqref{duh1}, furthermore yields that
\begin{equation}\label{duhit}
\begin{aligned}
u^K(t_n+\tau) & = e^{i \tau \Delta} u^K(t_n) \\
& \qquad - i e^{i \tau\Delta} \Pi_k\int_0^\tau \e^{- i s \Delta} \Big[\Big( e^{i s \Delta} \Pi_K u^K(t_n) - i e^{i s\Delta} \Pi_K T(\Pi_K u^K) (s,t_n)\Big)^2 \\
&\hskip3cm\cdot \Big( e^{-i s \Delta} \Pi_K \overline u^K(t_n) + i e^{-i s\Delta} \Pi_K \overline{T}(\Pi_K u^K)(s,t_n)\Big) \Big]\ddo s\\
& =  e^{i \tau \Delta} u^K(t_n) - i e^{i \tau \Delta} \Pi_K\int_0^\tau e^{-i s\Delta} \Big[ \left(e^{is\Delta} \Pi_K u^K(t_n)\right)^2 \left(e^{-i s\Delta}\Pi_K \overline u^K(t_n)\right) \Big]\mathrm{d}s \\
&\qquad + i e^{i \tau \Delta} \Pi_K \int_0^\tau  \black e^{-is\Delta} \Big[ T_1+T_2 + T_3 + T_4 + T_5\Big](s,t_n)\dd s \black
\end{aligned}
\end{equation}
with
\begin{equation}
\begin{aligned}\label{Tj}
& T_1(s,t_n) = \black -i \black \left(e^{is \Delta} \Pi_K u^K(t_n) \right)^2 e^{-is \Delta} \Pi_K \overline{T}(\Pi_K u^K)(s,t_n) \\
& T_2(s,t_n) = \black -2 \black \left( e^{is \Delta} \Pi_K u^K(t_n)\right) \left \vert e^{i s \Delta} \Pi_K T(\Pi_K u^K)(s,t_n)\right\vert^2 \\
& T_3(s,t_n) = \black i \black \left \vert e^{i s \Delta} \Pi_K T(\Pi_K u^K)(s,t_n)\right\vert^2 e^{i s \Delta} \Pi_K T(\Pi_K u^K)(s,t_n)\\
& \black T_4(s,t_n) = 2i \left|e^{is \Delta} \Pi_K u^K(t_n) \right|^2 e^{is \Delta} \Pi_K T(\Pi_K u^K)(s,t_n) \black \\
& \black T_5(s,t_n) = \left( e^{-is \Delta} \Pi_K \overline u^K(t_n)\right) \left( e^{i s \Delta} \Pi_K T(\Pi_K u^K)(s,t_n)\right)^2. \black
\end{aligned}
\end{equation}
In the following we set
\begin{equation}\label{err1}
E_1(u^K,\tau,t_n) = i  \int_0^\tau e^{-is\Delta}\Big[ T_1+T_2 + T_3 + T_4 + T_5\Big](s,t_n)\dd s
\end{equation}
such that by \eqref{duhit} we have that
\begin{equation}\label{duh2}
\begin{aligned}
u^K(t_n+\tau) = e^{i \tau \Delta} u^K(t_n) & - i e^{i \tau \Delta} \Pi_K \int_0^\tau e^{-i s \Delta} \Big[ \left(\e^{is\Delta} \Pi_K \black u^K(t_n)\right)^2 \left(e^{-i s\Delta}  \Pi_K \black \overline u^K(t_n)\right) \Big]\dd s \\
& + e^{i \tau \Delta}\Pi_K E_1(u^K,\tau,t_n).
\end{aligned}
\end{equation}
To compare the exact solution \eqref{duh2} with the numerical solution \eqref{scheme} we need the following Lemma.

\begin{lemma}\label{lem:locE3}
It holds that
\begin{multline}
e^{- i s \Delta}  \left( e^{i s \Delta} w\right)^2 \left( e^{-i s \Delta} \overline w\right) - w^2  \left( e^{-2 i s \Delta} \overline w\right) \\
= \blue - \black 2 i \int_0^s e^{-i s_1 \Delta} \left[ \nabla  \left( e^{i s_1 \Delta} w\right)^2 \nabla \left( e^{ i( s_1 -2s) \Delta } \overline w\right) + \left( \nabla e^{i s_1 \Delta} w\right)^2 \left( e^{ i( s_1 -2s) \Delta } \overline w\right) \right]\ddo {s_1},
\end{multline}
where we set $\nabla f \nabla g = \sum_{i = 1}^d (\partial_{i} f)(\partial_i g)$ and $(\nabla f)^2 = \nabla f \nabla f$.
\black
\end{lemma}

\begin{proof}
With the aid of the (inverse) Fourier  transform \black
$$
 w(x) = (2\pi)^{-d/2}\int_{\mathbb{R}^d} \widehat w (\xi) \, e^{i \langle x,\xi\rangle}\dd\xi \black
$$
we obtain with the notation $\xi_j \xi_\ell = \langle \xi_j, \xi_\ell\rangle$ \black that
\begin{equation}
\begin{aligned}
\mathcal{F}\Bigl( - \black 2 i \int_0^s & e^{-i s_1 \Delta} \left[ \nabla \left( e^{i s_1 \Delta} w\right)^2 \nabla \left( e^{ i( s_1 -2s) \Delta } \overline w\right) + \left( \nabla e^{i s_1 \Delta} w\right)^2 \left( e^{ i( s_1 -2s) \Delta } \overline w\right)\right]\dd {s_1}\Bigr)(\xi)\\
& =  2 i\, (2\pi)^{-3d/2} \int_{\xi_1,\xi_2,\xi_3} \delta_{\xi_1+\xi_2+\xi_3 = \xi}  \,\overline{\widehat w}(\xi_1) \widehat w (\xi_2) \widehat w(\xi_3)\, e^{ 2 i s \xi_1^2}  \times \black \\
&  \qquad\qquad \int_0^s \left( - \xi_1  (\xi_2+\xi_3)  + \xi_2 \xi_3\right) e^{i s_1 (-\xi_1 + \xi_2+\xi_3)^2} e^{- i s_1 (\xi_1^2+\xi_2^2+\xi_3^2)}\dd s_1 \black \\
&  =   (2\pi)^{-3d/2}  \int_{\xi_1,\xi_2,\xi_3} \delta_{\xi_1+\xi_2+\xi_3 = \xi}  \,\overline{\widehat w}(\xi_1) \widehat w (\xi_2) \widehat w(\xi_3)\, e^{ 2 i s \xi_1^2} \times \black \\
& \qquad\qquad\int_0^s 2 i  \left( - \xi_1  (\xi_2+\xi_3)  + \xi_2 \xi_3\right)e^{2 i s_1 (-\xi_1 (\xi_2+\xi_3) +  \xi_2 \xi_3) }\dd {s_1} \black \\
& = \mathcal{F}\Bigl( \mathrm{e}^{- i s_1\Delta}  \left(\mathrm{e}^{i s_1 \Delta} w\right)^2 \left(\mathrm{e}^{i (s_1-2s)\Delta  } \overline w\right)\Big \vert_{s_1 = 0}^s\Bigr)(\xi).
\end{aligned}
\end{equation}
This proves the desired relation. \black
\end{proof}

With the aid of the above lemma we \black get an alternative expression of the \black exact solution \eqref{duh2}.

\begin{cor}\label{cor:locE} The solution of \eqref{NLSK} can be expressed as follows
\begin{multline}\label{efinal}
u^K(t_{n+1}) =  e^{i \tau \Delta} u^K(t_{n})-\tau S_{K}(\tau) \left( \left(\Pi_K u^{K}(t_{n})\right)^2 \varphi_1(-2i \tau\Delta) \Pi_K \overline u^K(t_{n})\right)\\
+ i S_K(\tau)\left( E_{1}(u^K, \tau, t_{n}) + E_{2}(u^K, \tau, t_{n})\right),
\end{multline}
where \black $S_K= \Pi_K \mathrm{e}^{i \tau \Delta}$ \black is defined in \eqref{SK}, $E_1$ given in \eqref{err1} and $E_2$ reads
\begin{multline}
\label{err2}
E_2 (u^K, \tau, t_n)  = \black -2 \black \int_0^\tau \int_0^s  e^{-i s_1 \Delta} \Big[ \nabla \left(e^{i s_1 \Delta} \Pi_K u^K(t_n) \right)^2 \nabla \left(e^{ i( s_1 -2s) \Delta } \overline{\Pi_K u^K(t_n)} \ \right) \\
+ \left( \nabla e^{i s_1 \Delta}  \Pi_K u^K(t_n) \right)^2 \left(e^{ i( s_1 -2s) \Delta } \overline{\Pi_K u^K(t_n)} \right)
\Big]\mathrm{d}{s_1}\mathrm{d}s.
\end{multline}
\end{cor}
\begin{proof}
The corollary follows by applying Lemma \ref{lem:locE3} in the integral in \eqref{duh2}.
\end{proof}

\section{Global error analysis}
\label{sectionlocal2}

Note that we can write our scheme \eqref{scheme} in the form
$$
u^{n+1}= e^{i \tau \Delta} u^n -\tau S_{K}(\tau) \left( \left(\Pi_K u^n\right)^2 \varphi_1(-2i \tau\Delta) \Pi_K \overline u^n\right),
$$
and that the exact solution $u^K(t)$ of the projected equation \black is given by \eqref{efinal}. Let $e^n = u^K(t_n) - u^n$ denote the error, i.e., the difference between numerical and exact solution. The errors thus satisfies the following recursion
\begin{multline*}
e^{n+1} =  e^{i \tau \Delta} e^n -\tau S_{K}(\tau) \left(  \left(\Pi_K u^{K}(t_{n})\right)^2 \varphi_1(-2i \tau\Delta) \Pi_K \overline u^K(t_{n}) - \left(\Pi_K u^n\right)^2 \varphi_1(-2i \tau\Delta) \Pi_K \overline u^n\right)\\
+ i S_K(\tau)\left( E_{1}(u^K, \tau, t_{n}) + E_{2}(u^K, \tau, t_{n})\right)
\end{multline*}
with $e^0=0$. \black
Therefore, by \black solving this recursion, \black we obtain that
\begin{multline}
\label{duhamelscheme}
e^{n}=\tau \sum_{k=0}^{n-1}  S_{K}(t_{n-k}) \left( \left(\Pi_K u^{K}(t_{k})\right)^2 \varphi_1(-2i \tau\Delta) \Pi_K \overline u^K(t_{k}) -  \left(\Pi_K u^k\right)^2 \varphi_1(-2i \tau\Delta) \Pi_K \overline u^k\right) \\
+ \black i \sum_{k=0}^{n-1}  S_{K}(t_{n-k}) \left(E_{1}(u^K, \tau, t_{k}) + E_{2}(u^K, \tau, t_{k})\right). \black
\end{multline}

Let us set
\beq
\label{err2bis}
\mathcal{F}_{1}^n= \sum_{k=0}^{n-1}  S_{K}(t_{n-k}) E_{1}(u^K, \tau, t_{k}), \quad  \mathcal{F}_{2}^n= \sum_{k=0}^{n-1}  S_{K}(t_{n-k}) E_{2}(u^K, \tau, t_{k}).
\eeq
Then, we have the following estimates

\begin{lemma}
\label{LemmaF1n}
 For  every $T>0$  and $(p,q)$ admissible with $p>2$, there exists  $C_{T}>0$ such that
  for every $K$, $\tau$ as in  Proposition \ref{propuKinfty}, with $\alpha <2$ in dimension $3$, we have the estimates
\begin{equation}
\label{estF1n}
\|\mathcal{F}_{1}^n \|_{l^p_{\tau, N}L^{q}} \lesssim ( K \tau^{1 \over 2})^{ 2 \over p} \tau C_{T}, \quad  \| \mathcal{F}_{1}^n \|_{l^p_{\tau, N}W^{1,q}} \lesssim  K ( K \tau^{1 \over 2})^{ 2 \over p} \tau C_{T}.
\end{equation}
\end{lemma}

The second part of the estimate \eqref{estF1n} is very rough, but will be enough for our purpose. Note that, by using Sobolev embedding, we deduce from the above estimates that  in dimension $3$, we have
\beq
\label{estF1n3D}
\| \mathcal{F}_{1}^n \|_{l^4_{\tau, N}L^4}
\lesssim  \| \mathcal{F}_{1}^n \|_{l^4_{\tau, N} W^{{1 \over 4}, 3}}
\lesssim   \| \mathcal{F}_{1}^n \|_{l^4_{\tau, N}  L^3}^{3 \over 4} \| \mathcal{F}_{1}^n \|_{l^4_{\tau, N} W^{1, 3}}^{1 \over 4}
\lesssim \tau K^{1 \over 4} ( K \tau^{1 \over 2})^{ 1 \over 2}.
\eeq

As we will see below, $\mathcal{F}^n_{1}$ is the best part of the error in the sense that the above estimates yield an error of order $\tau$ in $l^\infty_{\tau}L^2$.

\begin{proof}
In the proof, $C_{T}$ will stand for a number that depends only on $T$ and on the estimates of Proposition \ref{propNLSK} of the exact solution. In particular, it is independent of $\tau$ and $K$. We first write by using the discrete Strichartz estimates
\begin{equation}
\label{EF1}
\begin{aligned}
\| \mathcal{F}_{1}^n \|_{l^p_{\tau, N}L^{q}} &\leq  ( K \tau^{1 \over 2})^{ 2 \over p} \tau^{-1} \| E_{1}(u^K, \tau, t_{k}) \|_{l^1_{\tau, N}L^2} \\
&\leq C_{T}( K \tau^{1 \over 2})^{ 2 \over p} \sup_{s\in[0, \tau]} \left(   \| T_{1}(t_{n}, s) \|_{l^1_{\tau, N}L^2}  +  \| T_{2}(t_{n}, s) \|_{l^1_{\tau,N}L^2}  +   \| T_{3}(t_{n}, s) \|_{l^1_{\tau,N}L^2}\right.\\
&\black \left. \hspace{6.2cm} {}+  \| T_{4}(t_{n}, s) \|_{l^1_{\tau,N}L^2}  +   \| T_{5}(t_{n}, s) \|_{l^1_{\tau,N}L^2}\right).\black
\end{aligned}
\end{equation}
Next, by using \eqref{Tj}, we get that
$$
\| T_{1}(t_{n}, s) \|_{L^2} \lesssim \| (e^{is \Delta } u^K(t_{n}))^2\|_{L^3} \| e^{i s \Delta } T(\Pi_K u^K)(s, t_{n}) \|_{L^6} \lesssim  \| e^{is \Delta } u^K(t_{n})\|_{L^6}^2 \| e^{i s \Delta } T(\Pi_K u^K)(s, t_{n}) \|_{L^6}.
$$
Next, we have  by Sobolev embedding that
$$
\| e^{is \Delta } u^K(t_{n})\|_{L^6} \lesssim   \| e^{is \Delta } u^K(t_{n})\|_{H^1} \lesssim \|u^K(t_{n}) \|_{H^1}
$$
and since
\begin{equation}
\label{TuuuDuhamel}
e^{i s \Delta } T(\Pi_K u^K)(s, t_{n})= \int_{0}^s e^{i(s-\tilde{s})\Delta } |\Pi_{K}u^K (t_{n} + \tilde{s})|^2 \Pi_{K}u^K (t_{n} + \tilde{s}) \dd \tilde{s},
\end{equation}
we obtain by Sobolev embedding that
$$
\|e^{i s \Delta } T(\Pi_K u^K)(s, t_{n})\|_{L^6} \lesssim  \black  \|T(\Pi_K u^K)(s, t_{n})\|_{H^1}. \black
$$%
Consequently,
\begin{equation}
\label{Tuuusup1}
\begin{aligned}
&\| T(\Pi_K u^K)(s, t_{n})\|_{H^1} \\
&  \lesssim  \int_{0}^\tau \left( \left\| |\Pi_{K}u^K (t_{n} + \tilde{s})|^2  \nabla \Pi_{K}u^K(t_{n} + \tilde{s}) \right\|_{ L^2} + \left\| |\Pi_{K}u^K (t_{n} + \tilde{s})|^2 \Pi_{K}u^K(t_{n} + \tilde{s}) \right\|_{ L^2} \right) \dd \tilde s  \\
&  \lesssim \| u^K \|_{L^\infty H^1}  \int_{0}^\tau \| u^K(t_{n}+ \tilde s)\|_{L^\infty}^2 \dd\tilde s
\end{aligned}
\end{equation}
which yields
\beq
\label{estimationTF1}
\|T(\Pi_K u^K)(s, t_{n})\|_{l^1_{\tau,N}H^1} \lesssim \tau \|u^K\|_{L^\infty_{T}H^1} \sup_{\tilde s \in [0, \tau]} \|u^K(t_{n}+\tilde s)\|_{l^2_{\tau,N}L^\infty}^2.
\eeq
We thus obtain that
$$
\| T_{1}(t_{n}, s) \|_{l^1_{\tau,N}L^2} \leq  \tau \| u^K \|_{L^\infty_{T} H^1}^3 \sup_{\tilde s \in [0, \tau]}\| u^K(t_{n}+ \tilde s)\|_{l^2_{\tau,N}L^\infty}^2.
$$
By using Proposition \ref{propuKinfty} that yields \eqref{uKinfty} thanks to Remark \ref{remcond} (with $\alpha <2$ in dimension $3$), we finally obtain
\beq
\label{T1fini}   \| T_{1}(t_{n}, s) \|_{l^1_{\tau,N}L^2} \leq \tau C_{T}.
\eeq

In a similar way, we obtain that
\begin{align*}  \|T_{2}(s, t_{n}) \|_{L^2} &  \lesssim  \| e^{is \Delta } u^K(t_{n})\|_{L^6} \| e^{i s \Delta } T(\Pi_K u^K)(s, t_{n}) \|_{L^6}^2 \\
& \lesssim \|u^K(t_{n})\|_{H^1} \|T(\Pi_K u^K)(s, t_{n}) \|_{H^1}^2
\end{align*}
and hence, by using again \eqref{estimationTF1}, we get
$$
\|T_{2}(s, t_{n}) \|_{l^1_{\tau,N}L^2} \lesssim \|u\|_{L^\infty_{T}H^1} \|T(\Pi_K u^K)(s,t_{n})\|_{l^\infty_{\tau,N}H^1} \|T(\Pi_K u^K)(s,t_{n})\|_{l^1_{\tau,N}H^1}.
$$
We can use again \eqref{estimationTF1} to estimate $ \|(\Pi_K u^K)(s,t_n)\|_{l^1_{\tau,N}H^1}$. Therefore, we only need to estimate $\|T(\Pi_K u^K)(s,t_n) \|_{l^\infty_{\tau,N}H^1}$. By using again \eqref{Tuuusup1} we get that
$$
\|T (\Pi_K u^K)(s,t_{n})\|_{H^1} \lesssim \int_{0}^\tau \|u^K(t_{n}+ s )\|_{H^1} \|u^K(t_{n}+ s) \|_{L^\infty}^2\dd s \lesssim \|u^K\|_{L^\infty_{T}H^1} \|u^K\|_{L^2_{T}L^\infty}^2
$$
and, therefore,
\beq
\label{TuuuSup2}
\|T(\Pi_K u^K)(s,t_{n})\|_{l^\infty_{\tau,N}H^1} \leq \|u^K\|_{L^\infty_{T}H^1} \|u^K\|_{L^2_{T}L^\infty}^2 \leq C_{T}
\eeq
since $u^K$ satisfies the continuous Strichartz estimates \eqref{borne2}. We thus finally obtain that
\beq
\label{T2fini}
\|T_{2}(s, t_{n}) \|_{l^1_{\tau,N}L^2} \lesssim \tau C_{T}.
\eeq

Finally, from the same arguments as above, we have that
$$
\|T_{3}(s, t_{n}) \|_{L^2} \lesssim   \| e^{i s \Delta } T(\Pi_K u^K)(s, t_{n}) \|_{L^6}^3  \lesssim  \| T(\Pi_K u^K)(s, t_{n}) \|_{H^1}^3.
$$
Consequently,
$$
\|T_{3}(s, t_{n}) \|_{l^1_{\tau,N}L^2} \lesssim   \| T(\Pi_K u^K)(s, t_{n}) \|_{l^\infty_{\tau,N}H^1}^2 \| T(\Pi_K u^K)(s, t_{n}) \|_{l^1_{\tau,N}H^1}
$$
and therefore, by using \eqref{TuuuSup2} and \eqref{estimationTF1}, we also obtain that
\beq
\label{T3fini}
\|T_{3}(s, t_{n}) \|_{l^1_{\tau,N}L^2}\leq \tau C_{T}.
\eeq
\black The term $T_4$ is estimated in the same way as $T_1$, the term $T_5$ in the same way as $T_2$. \black
Consequently, by combining \eqref{T1fini}, \eqref{T2fini}, \eqref{T3fini} with \eqref{EF1}, we finally obtain that
$$
\| \mathcal{F}_{1}^n \|_{l^p_{\tau, N}L^{q}} \lesssim ( K \tau^{1 \over 2})^{ 2 \over p}\tau C_{T}.
$$
Since $\mathcal{F}_{1}^n = \Pi_{2K} \mathcal{F}_{1}^n$ we \black also readily obtain that \black 
$$
\| \mathcal{F}_{1}^n \|_{l^p_{\tau, N}W^{1,q}} \lesssim K \| \mathcal{F}_{1}^n \|_{l^p_{\tau, N}L^{q}} \lesssim K ( K \tau^{1 \over 2})^{ 2 \over p} \tau C_{T}.
$$
 Indeed,  the first above estimate, is a consequence of the fact that we can write 
 \beq
 \label{expressionfinale} \Pi_{2K} \mathcal{F}_{1}^n = \rho_{\epsilon} * f, \quad \rho_{\epsilon} (x)= { 1 \over \epsilon^d} \rho \left ({x\over \epsilon} \right),  \quad \epsilon= { 1 \over 2 K}, \quad \rho= \mathcal{F}^{-1}( \chi^2) \in \mathcal{S}(\mathbb{R}^d)\eeq and standard convolution inequalities that thus  yield
  $$ \| \nabla  \mathcal{F}_{1}^n\|_{L^q} \lesssim { 1 \over \epsilon }   \|\mathcal{F}_{1}^n\|_{L^q}.$$
This ends the proof of \eqref{estF1n}.
\end{proof}

We shall now analyze the second part of the error.

\begin{lemma}
\label{LemmaF2n}
For every $T>0$ and $(p,q)$ admissible with $p>2$, there exists $C_{T}>0$ such that for every $K$, $\tau$ as in Proposition \ref{propuKinfty}, with $\alpha <2$ in dimension $3$, we have the estimates
\begin{alignat}{2}
& \label{F2n1D}
\| \mathcal{F}_{2}^n \|_{l^p_{\tau, N}L^{q}} \leq C_{T} \,\tau (K \tau^{1 \over 2})^{{1 \over 2}+ {2 \over p} }, \quad &&\mbox{if $d= 1$},  \\
& \label{F2n2D}
\| \mathcal{F}_{2}^n \|_{l^p_{\tau, N}L^{q}} \leq  C_{T} \, \tau (K \tau^{1 \over 2})^{ 1 + {2 \over p} },  \quad &&\mbox{if $d= 2$},  \\
& \label{F2n3D} \| \mathcal{F}_{2}^n \|_{l^p_{\tau, N}L^{q}} \leq C_{T}\,  \tau (K \tau^{1 \over 2})^{ 2 + {2 \over p}  }(\log K)^{ 2\over 3}, \quad &&\mbox{if $d= 3$}.
\end{alignat}
Moreover, in dimension $3$, we also have the estimate
\beq
\label{3D+} \| \mathcal{F}_{2}^n \|_{l^4_{\tau, N}L^{4}} \leq C_{T}\,  \tau (K \tau^{1 \over 2})^{ 5\over 2  } K^{1 \over 4}(\log K)^{2 \over 3}, \quad \mbox{if $d= 3$}.
\eeq
\end{lemma}

\begin{proof}
At first, we observe that using the expressions \eqref{err2}, \eqref{err2bis}, we can write that
\begin{equation}
\label{1F2n0}
\begin{aligned}
\mathcal{F}_{2}^n &= 2 \int_{0}^{\tau} \int_{0}^s e^{-is_{1} \Delta } \sum_{k=0}^{n-1}  S_{K}(t_{n-k})  G(s, s_{1}, t_{k})\dd s_{1}\dd s \\
& =  2 \int_{0}^{\tau} \int_{0}^s \sum_{k=0}^{n-1}  S_{K}(t_{n-k}- s_{1})  G(s, s_{1}, t_{k})\dd s_{1}\dd s,
\end{aligned}
\end{equation}
where
\begin{multline*}
G(s, s_{1}, t_{k}) = \black - \black \nabla  \left( e^{i s_1 \Delta} \Pi_K u^K(t_k) \right)^2 \nabla \left( e^{ i( s_1 -2s) \Delta } \overline{\Pi_K u^K(t_k)} \right) \\
+   \left( \nabla e^{i s_1 \Delta}  \Pi_K u^K(t_k) \right)^2 \left( e^{ i( s_1 -2s) \Delta } \overline{\Pi_K u(t_k)} \right)
\end{multline*}
and we observe that $s/\tau$, $s_{1}/\tau$, $(s_{1}- 2s) /\tau$ are uniformly bounded in $[-2, 1]$ so that we will be able to use Theorem \ref{theoDSE} and Propositions \ref{propuKinfty} and \ref{uKgrid}. We first estimate
\begin{equation}
\label{1F2n}
\| \mathcal{F}_{2}^n \|_{l^p_{\tau, N}L^{q}} \lesssim \tau^2  \sup_{0 \leq  s_{1} \leq s \leq  \tau}\left\|\sum_{k=0}^{n-1}  S_{K}(t_{n-k} - s_1)  G(s, s_{1}, t_{k})
\right\|_{l^p_{\tau, N}L^{q}} .
\end{equation}
Then, using  discrete Strichartz estimates, we obtain that
$$
\| \mathcal{F}_{2}^n \|_{l^p_{\tau, N}L^{q}} \lesssim \tau (K \tau^{1 \over 2})^{{ 2 \over p}}\sup_{0 \leq s, s_{1} \leq \tau}\ \|G(s, s_{1}, t_{k}) \|_{l^{ 1}_{\tau, N} L^{2}}.
$$
We shall then use slightly different arguments depending on the dimension. In dimension $d \leq 2$, we use H\"older's inequality to get
\begin{align*}
\|G(s, s_{1}, t_{k}) \|_{L^2} & \lesssim  \| \nabla e^{-i(s_{1}- 2s) \Delta} \Pi_{K}u^K(t_{k})\|_{L^4}   \| \nabla e^{is_{1}\Delta } \Pi_{K}u^K(t_{k})\|_{L^4}
\|e^{is_{1}\Delta }\Pi_{K} u^K(t_{k})\|_{L^\infty}  \\
& \quad \quad  +  \| \nabla e^{is_{1}\Delta }\Pi_{K} u^K(t_{k})\|_{L^4}^2 \|e^{-i(s_{1}- 2 s)\Delta }\Pi_{K} u^K(t_{k})\|_{L^\infty}
\end{align*}
and therefore,
\beq
\label{F21}
\| \mathcal{F}_{2}^n \|_{l^p_{\tau, N}L^{q}} \lesssim \tau (K \tau^{1 \over 2})^{  { 2 \over p} } \left(\sup_{\hat s \in [-2\tau, \tau]}  \| \nabla e^{i \hat s\Delta }\Pi_{K} u^K(t_{k})\|_{l^4_{\tau, N}L^4}\right)^2 \sup_{\hat s \in [-2\tau, \tau]} \|e^{i \hat s\Delta }\Pi_{K} u^K(t_{k})\|_{l^2_{\tau, N}L^\infty}.
\eeq
Next, we use the estimate
\beq
\label{F22}
\sup_{\hat s \in [-2\tau, 2\tau]} \|e^{i \hat s\Delta }\Pi_{K} u^K(t_{k})\|_{l^2_{\tau, N}L^\infty} \leq C_{T}
\eeq
from \eqref{Linftypq} of Proposition \ref{propuKinfty}. Indeed, as noticed after Proposition \ref{propuKinfty}, in \black dimensions $1$ and 2, \black this estimate is true without further restriction on $\alpha \geq 1$. \black Moreover, for all $\hat s \in [-2\tau, 2\tau]$, \black we have the estimate
\beq
\label{F23}
\| \nabla e^{i \hat s\Delta }\Pi_{K} u^K(t_{k})\|_{l^4_{\tau, N}L^4}\leq C_{T}  (K \tau^{1 \over 2})^{ d \over 4}
\eeq
from Proposition \ref{uKgrid} in dimension $d \leq 2$. Indeed for $d=1$, using H\"older and \eqref{nablauKgrid}, we have
$$
\| \nabla e^{i \hat s\Delta }\Pi_{K} u^K(t_{k})\|_{l^4_{\tau, N}L^4}\leq C_{T} \| \nabla e^{i \hat s\Delta }\Pi_{K} u^K(t_{k})\|_{l^8_{\tau, N}L^4} \leq C_{T}  (K \tau^{1 \over 2})^{ 1 \over 4},
$$
while for $d=2$, we can use directly the fact that $(4,4)$ is an admissible Strichartz pair to get
$$
\| \nabla e^{i \hat s\Delta }\Pi_{K} u^K(t_{k})\|_{l^4_{\tau, N}L^4}\leq C_{T} (K \tau^{1 \over 2})^{ 1 \over 2}.
$$
Consequently, by combining \eqref{F21}, \eqref{F22}, \eqref{F23}, we get the desired estimate
$$
\| \mathcal{F}_{2}^n \|_{l^p_{\tau, N}L^{q}} \lesssim \tau (K \tau^{1 \over 2})^{  { 2 \over p} } (K \tau^{1 \over 2})^{ d \over 2}
$$
for $d \leq 2$.

In dimension $3$,  the estimate   \eqref{F21}  is not sufficient to conclude since $(4,4)$ is not an admissible pair. We write in place the estimate
$$
\| \mathcal{F}_{2}^n \|_{l^p_{\tau, N}L^{q}} \lesssim \tau (K \tau^{1 \over 2})^{  { 2 \over p} } \left(\sup_{\hat s \in [-2\tau, \tau]}  \| \nabla e^{i \hat s\Delta }\Pi_{K} u^K(t_{k})\|_{l^{{8 \over 3}}_{\tau, N}L^4}\right)^2 \sup_{\hat s \in [-2\tau, \tau]} \|e^{i \hat s\Delta }\Pi_{K} u^K(t_{k})\|_{l^4_{\tau, N}L^\infty}
$$
and therefore, we get from \eqref{nablauKgrid} that
$$
\| \mathcal{F}_{2}^n \|_{l^p_{\tau, N}L^{q}} \lesssim \tau (K \tau^{1 \over 2})^{  { 2 \over p}  + {3 \over 2}} C_{T} \sup_{\hat s \in [-2\tau, \tau]} \|e^{i \hat s\Delta }\Pi_{K} u^K(t_{k})\|_{l^4_{\tau, N}L^\infty}.
$$
Here we cannot use anymore Proposition  \ref{propuKinfty}  in order to estimate $\|e^{i \hat s\Delta }\Pi_{K} u^K(t_{k})\|_{l^4_{\tau, N}L^\infty}$ without loss \black unless we take \black $\alpha =1$, which would yield a non optimal total error. We are thus forced to use Sobolev embedding and  \eqref{nablauKgrid}. \black Thanks to Lemma \ref{lemsobfin} \black 
$$
\|e^{i \hat s\Delta }\Pi_{K} u^K(t_{k})\|_{l^4_{\tau, N}L^\infty} \leq \log K  \|e^{i \hat s\Delta }\Pi_{K} u^K(t_{k})\|_{l^4_{\tau, N}W^{1, 3}} \leq C_{T}(\log K)^{2 \over 3} (K \tau^{1 \over 2})^{ 1 \over 2}.
$$
This finally yields
$$
\| \mathcal{F}_{2}^n \|_{l^p_{\tau, N}L^{q}} \lesssim \tau (K \tau^{1 \over 2})^{  { 2 \over p} } (K \tau^{1 \over 2})^{2} C_{T} (\log K)^{ 2\over 3},
$$
which is the desired estimate in dimension $3$.

To get \eqref{3D+}, we just observe that since $\Pi_{2K} \mathcal{F}_{2}^n = \mathcal{F}_{2}^n$,
 we can thus write $ \mathcal{F}_{2}^n = \rho_{\epsilon} *  \mathcal{F}_{2}^n$ with $\rho_{\epsilon}$ as in \eqref{expressionfinale}
  and use Young’s inequality to obtain 
  \black 
$$
\|\mathcal{F}_{2}^n \|_{l^4_{\tau,N}L^4} \lesssim K^{1 \over 4} \|\mathcal{F}_{2}^n \|_{l^4_{\tau,N}L^3}.
$$
Since $(4,3)$ is an admissible pair in dimension $3$, we can use \eqref{F2n3D} to get the desired estimate.
\end{proof}

\section{Proof of Theorem \ref{maintheo}}
\label{sectionfinale}

At first, we use Corollary \ref{corNLSK}, to write that
\begin{equation}
\label{erreurfinal}
\|u(t_{n})- u^{n}\|_{L^2} \leq \|u(t_{n}) - u^K(t_{n})\|_{L^2} + \|u^K(t_{n}) - u^{n}\|_{L^2} \leq {C_{T} \over K} + \|e^{n}\|_{L^2}.
\end{equation}
To estimate $e^{n}$ we shall use equation \eqref{duhamelscheme}. Note that the consistency error on the right-hand side can be estimated thanks to Lemma \ref{LemmaF1n} and Lemma \ref{LemmaF2n}. We shall choose our parameter $K$ in an optimal way so that the contribution of the consistency error in $L^2$ is of order $1/K$ in order to get contributions of the same order in the two terms of \eqref{erreurfinal}. This choice will depend on the dimension since the estimates of Lemma \ref{LemmaF2n} depend on the dimension.

\subsection*{Dimension $d \leq 2$}

In dimension $d \leq 2$, by using Lemma \ref{LemmaF1n} and Lemma \ref{LemmaF2n}, we have that
$$
\|\mathcal{F}_{1}^n\|_{l^\infty_{\tau, N}L^2} +  \|\mathcal{F}_{2}^n\|_{l^\infty_{\tau, N}L^2} \leq C_{T} (\tau + K^{d\over 2} \tau^{ 1 + {d\over 4}}).
$$
We thus choose $K$ such that $K^{d \over 2} \tau^{ 1 + {d \over 4}}= {1 \over K}$ which gives
\beq
\label{choix1}
K= \tau^{-{  1\over 2} {4 + d \over 2 + d }}.
\eeq
Note that this choice gives in particular that
\beq
\label{Ktau12}
K \tau^{1 \over 2}=  \tau^{-{ 1 \over 2 +d}},
\eeq
an expression that will be useful in future computations. Under this-CFL type condition, we get that
$$
\|\mathcal{F}_{1}^n\|_{l^\infty_{\tau, N}L^2} +  \|\mathcal{F}_{2}^n\|_{l^\infty_{\tau, N}L^2} \leq C_{T}\tau^{{1\over 2}{4+ d \over 2+d}}
$$
and more generally that for every $(p,q)$ admissible, $p>2$,
\beq
\label{1dcons1}
\|\mathcal{F}_{1}^n\|_{l^p_{\tau, N}L^q} +  \|\mathcal{F}_{2}^n\|_{l^p_{\tau, N}L^q} \leq  C_{T}\tau^{{1\over 2}{1 \over 2+d}( 4+d - {4\over p}) }.
\eeq

Let us define  $N_{1}$  such that $N_{1} \tau = T_{1} \leq T$. We shall first prove by induction that $e^{n}$ verifies the estimate
\beq
\label{induc1}
\|e^n\|_{X_{\tau,k}}:=  { 1\over \tau^{{1\over 2}{4+ d \over 2+d}} }\|e^n\|_{l^\infty_{\tau, k}L^2} +  {1 \over \tau^{{1\over 4}{6+d \over 2+d} }} \| e^n \|_{l^{8\over d}_{\tau, k}L^4} \leq 8 C_{T}, \qquad  0 \leq  k \leq N_{1},
\eeq
for $T_{1}$ and $\tau$ sufficiently small compared to $C_{T}$. Note that the control of the above norm gives that we propagate an estimate of order $\tau^{ {1\over 4}{6+d \over 2+d} }$ for  the norm $ \| e^n \|_{l^{8\over d}_{\tau, k}L^4} $. This is less than $\tau^{{1\over 4}{8+d \over 2+d} }$  that one would expect in view of estimate \eqref{1dcons1}. This would nevertheless be sufficient to close the following argument. One of the \black reasons for this choice is the control of \black terms involving the filter function $ \varphi_{1} (2i \tau \Delta)$. Indeed, this operator is not uniformly bounded on $L^p$ for $p \neq 2$. Nevertheless, we get by Sobolev embedding  and \eqref{estfilter3} that
$$
\| \black \varphi_{1} (-2i \tau \Delta) e^n \black \|_{l^{8\over d}_{\tau, k}L^4}  \lesssim \| \varphi_{1} (-2 i \tau \Delta) e^n \|_{l^{8\over d}_{\tau, k}H^{d \over 4}}
\black \leq  C_{T}{ 1 \over \tau^{d \over 8}} \| e^n \|_{l^{\infty}_{\tau, k}L^2} \black 
\leq C_{T}{ \tau^{{1\over 2}{4+ d \over 2+d}} \over  \tau^{d \over 8}   } \| e^n \|_{X_{\tau,k}}.
$$
Consequently, since $ {1\over 2}{4+ d \over 2+d} -{d \over 8}  \geq {1\over 4}{6+d \over 2+d}$ when $d \leq 2$, we get that
\beq
\label{enfilter}
\| \varphi_{1} \black (-2i \tau \Delta) \black e^n \|_{X_{\tau,k}} \leq C_{T} \|e^n \|_{X_{\tau,k}}.
\eeq
Let us  rewrite \eqref{duhamelscheme} as
\beq
\label{duhamelscheme2}
e^{n}=\tau \sum_{k=0}^{n-1}  S_{K}(t_{n-k}) G_{k}+  \mathcal{F}_{1}^n + \mathcal{F}_{2}^n
\eeq
where
$$
G_{k} = \Pi_K e^k\bigl(\Pi_{K} u^K(t_{k}) + \Pi_{K} u^k\bigr) \varphi_1(-2i \tau\Delta) \Pi_K \overline u^K(t_{k})+\left(\Pi_K u^k\right)^2 \varphi_1(-2i \tau\Delta) \Pi_K \overline e^k.
$$
Note that by substituting $u^k = u^K(t_{k}) - e^{k}$, we can write
\beq
\label{Tkdef}
G_{k}=  G_{k}^1 + G_{k}^2 + G_{k}^3,
\eeq
where
\begin{align*}
G_{k}^1 &=  2( \Pi_{K}u^K(t_{k}) )( \varphi_1(-2i \tau\Delta) \Pi_K \overline u^K(t_{k})) ( \Pi_{K} e^{k}) + (\Pi_{K} u^K(t_{k}))^2  \varphi_1(-2i \tau\Delta) \Pi_K \overline{e}^k, \\
G_{k}^2 &= -\bigl(\varphi_1(-2i \tau\Delta) \Pi_K \overline u^K(t_{k})\bigr)(\Pi_{K} e^k)^2 - 2 (\Pi_{K} u^K(t_{k}))( \Pi_{k} e^{k})\varphi_1(-2i \tau\Delta)\Pi_K \overline{e}^k, \\
G^3_{k} &=  (\Pi_{K} e^k)^2 \varphi_1(-2i \tau\Delta) \Pi_K \overline{e}^k.
\end{align*}

To estimate $e^n$, we use the discrete Strichartz inequalities of Theorem \ref{theoDSE} and our choice \eqref{choix1}. In the following $C$ is again a generic number independent of $T_{1}$, $T$, $\tau$ and $K$. We first get that
\beq
\label{esten1}
\|e^n\|_{l^\infty_{\tau,k+1}L^2} \leq  C_{T} \tau^{{1\over 2}{4+ d \over 2+d}} + C \|G_{n}^1\|_{l^1_{\tau, k}L^2} + C \|G_{n}^2\|_{l^1_{\tau, k}L^2} + C { 1 \over \tau^{{1 \over 4} { d \over 2+ d}}}  \|G_{n}^3 \|_{l^{\left(8 \over d\right)'}_{\tau, k}L^{4\over 3}}.
\eeq
To estimate the right-hand side, we first use that
$$
\|G_{n}^1\|_{l^1_{\tau, k}L^2} \leq C \| e^n\|_{l^\infty_{\tau,k}L^2} \left( \|u^K(t_{n})\|_{l^2_{\tau,k} L^\infty}^2 +\|\varphi_1(-2i \tau\Delta) \Pi_K u^K(t_{n})\|_{l^2_{\tau,k} L^\infty}^2\right).
$$
If $d=1$, the above right-hand side can be easily estimated since
$$
\|u^K(t_{n})\|_{l^2_{\tau,k} L^\infty}^2 +\|\varphi_1(-2i \tau\Delta) \Pi_K u^K(t_{n})\|_{l^2_{\tau,k} L^\infty}^2 \leq C T_{1} \|u^K\|_{L^\infty_{T_{1}} H^1}^2 \leq \black T_{1}C_{T}^2. \black
$$
If $d=2$, we can use Remark \ref{remcond} to obtain
\begin{multline*}
\|u^K(t_{n})\|_{l^2_{\tau,k} L^\infty} +\|\varphi_1(-2i \tau\Delta) \Pi_K u^K(t_{n})\|_{l^2_{\tau,k} L^\infty} \\
\leq T_{1}^{1 \over 4}\left( \|u^K(t_{n})\|_{l^4_{\tau,k} W^{\sigma, 4}} + \|\varphi_1(-2i \tau\Delta) \Pi_K u^K(t_{n})\|_{l^4_{\tau,k} W^{\sigma, 4}}\right)  \black \leq  T_{1}^{1 \over 4} C_T \black
\end{multline*}
for some suitable choice of $\sigma$ slightly larger than $1/2$. This thus yields, by using Proposition \ref{propuKinfty} and Corollary \ref{corfilter},
\beq
\label{Gn1}
\|G_{n}^1\|_{l^1_{\tau, k}L^2}\leq \black T_{1}^{1\over 2} C_{T}^2 \black \|e^n\|_{l^\infty_{\tau, k}L^2}.
\eeq
Let us now estimate $G_{n}^2$. From similar arguments, we obtain that
$$
\|G_{n}^2\|_{l^1_{\tau, k}L^2} \leq C \Bigl(\| e^n\|_{l^4_{\tau,k}L^4}^2 +  \| \varphi_{1}(-2 i \tau\Delta )e^n\|_{l^4_{\tau,k}L^4}^2\Bigr) \left( \|u^K(t_{n})\|_{l^2_{\tau,k} L^\infty}
+\|\varphi_1(-2i \tau\Delta) \Pi_K u^K(t_{n})\|_{l^2_{\tau,k} L^\infty} \right)
$$
which yields
\beq
\label{Gn2}
\|G_{n}^2\|_{l^1_{\tau, k}L^2}\leq  C_{T}(\|e^n\|_{l^{8\over d }_{\tau, k}L^4}^2 + \|\varphi_1(-2i \tau\Delta)e^n\|_{l^{8\over d }_{\tau, k}L^4}^2).
\eeq
Finally, to estimate the last term in the right-hand side of \eqref{esten1}, we use that
$$
\|G_{n}^3 \|_{l^{\left(8 \over d\right)'}_{\tau, k}L^{4\over 3}} \leq  C\left(  \left\| \|e^n\|_{L^4}^3\right\|_{l^{\left(8 \over d\right)'}_{\tau,k}} + \left\| \|\varphi_{1}(-2i \tau \Delta)e^n\|_{L^4}^3\right\|_{l^{\left(8 \over d\right)'}_{\tau,k}} \right)
$$
and, \black since $3 {\left( 8 \over d\right)'} = {24 \over 8-d}\leq {8\over d}$ for $d \leq 2$, \black we obtain from  H\"older that
\beq
\label{Gn3}
\|G_{n}^3 \|_{l^{\left(8 \over d\right)'}_{\tau, k}L^{4\over 3}}  \leq  C_{T}\left (\|e^n\|_{l^{8\over d }_{\tau, k}L^4}^3 \black  + \left\|\varphi_{1}(-2i \tau \Delta)e^n \right\|_{l^{8\over d }_{\tau, k}L^4}^3 \right). \black
\eeq
Consequently, by plugging \eqref{Gn1}, \eqref{Gn2} and \eqref{Gn3} into \eqref{esten1} and by using the observation \eqref{enfilter}, we get that
\beq
\label{esten2}
\|e^n\|_{l^\infty_{\tau,k+1}L^2} \leq  C_{T} \tau^{{1\over 2}{4+ d \over 2+d}} + \black T_{1}^{1\over 2} C_{T}^2 \black \|e^n\|_{l^\infty_{\tau, k}L^2} +C_{T} \tau^{ {1 \over 2} { 6 +d \over 2 +d}} \|e^n\|_{X_{\tau,k}}^2 + C_{T}  { 1 \over \tau^{ {1 \over 4} {d \over 2 + d}} } \tau^{{ {3 \over 4} { 6 +d \over 2 +d}}}\|e^n\|_{X_{\tau,k}}^3.
\eeq
In a similar way, by using again the discrete Strichartz inequalities, we find that
$$
\| e^n \|_{l^{8\over d}_{\tau, k+1}L^4} \leq  C_{T} \tau^{{1\over 4}{8+d \over 2+d} } +    { C \over \tau^{ {1 \over 4} {d \over 2 + d}}}(\|G_{n}^1\|_{l^1_{\tau, k}L^2} + C \|G_{n}^2\|_{l^1_{\tau, k}L^2})  + { C \over \tau^{ {1 \over 2} {d \over 2 + d}} }\|G_{n}^3 \|_{l^{\left(8 \over d\right)'}_{\tau, k}L^{4\over 3}}.
$$
Consequently, by using again \eqref{Gn1}, \eqref{Gn2}, \eqref{Gn3} and \eqref{enfilter}, we find that
\beq
\label{esten3}
\| e^n \|_{l^{8\over d}_{\tau, k+1}L^4} \leq  C_{T} \tau^{{1\over 4}{8+d \over 2+d} } + { 1 \over \tau^{ {1 \over 4} {d \over 2 + d}}}  \black T_{1}^{1\over 2} C_{T}^2 \black \|e^n\|_{l^\infty_{\tau, k}L^2} + { C_{T} \over \tau^{ {1 \over 4} {d \over 2 + d}}} \tau^{ {1 \over 2} { 6 +d \over 2 +d}} \|e^n\|_{X_{\tau,k}}^2 + { C \over \tau^{ {1 \over 2} {d \over 2 + d}} }   \tau^{{ {3 \over 4} { 6 +d \over 2 +d}}}\|e^n\|_{X_{\tau,k}}^3.
\eeq
By combining \eqref{esten3}, \eqref{esten2} and by using that  $\|e^n\|_{X_{\tau,k}}$ satisfies \eqref{induc1}, we obtain that
$$
\|e^n\|_{X_{\tau,k}} \leq 2C_{T}  + \black T_{1}^{1 \over 2} C_T^2 \black \|e^n\|_{X_{\tau,k}} + C C_{T}^3  \tau^{1 \over 2 + d }  +  C C_{T}^3 \tau^{ 5 \over 2 ( 2 + d)}
\leq 2C_{T}  + \black T_{1}^{1 \over 2} C_T^2 \black  \|e^n\|_{X_{\tau,k}} + C C_{T}^3  \tau^{{1 \over 2 + d }{7\over 2}}.
$$
Consequently, by taking $T_{1}$ sufficiently small so that \black $T_{1}^{1 \over 2} C_T^2 \leq {1 \over 2}$, \black we get that
$$
\|e^n\|_{X_{\tau,k}} \leq 4C_{T}  +  C C_{T}^3  \tau^{{1 \over 2 + d }{7\over 2}} \leq 8 C_{T}
$$
for $ \tau $ sufficiently small. This proves that
$$
\|e^n\|_{X_{\tau,N_{1}}} \leq 8 C_{T}.
$$
We can then iterate the  estimates on $[T_{1}, 2T_{1}], ...$ to finally obtain after a finite number of steps
$$
\|e^n\|_{X_{\tau,N}} \leq  \widetilde{C}_{T}.
$$
This proves the error estimate in dimension $d \leq 2$.

 \subsection*{Dimension $d= 3$}
 
\black For $d= 3$, following \black the same scheme of proof, we observe that
$$
\|\mathcal{F}_{1}^n\|_{l^\infty_{\tau, N}L^2} +  \|\mathcal{F}_{2}^n\|_{l^\infty_{\tau, N}L^2} \leq C_{T} (\tau + \tau^2 K^2(\log K)^{2 \over 3}).
$$
To optimize the total error, we thus choose $K$ such that $ \tau^2 K^2= {1 \over K}$ which yields
\beq
\label{choix3}
K= \tau^{-{2 \over 3}}
\eeq
and therefore
$$ 
\alpha = { 4 \over 3}<2, \quad K \tau^{1 \over 2}= \tau^{-{1 \over 6}}.
$$
The error thus verifies in particular  thanks to Lemmas \ref{LemmaF1n}, \ref{LemmaF2n} and \eqref{estF1n3D} that
\beq
\label{erreur3d}  
\|\mathcal{F}_{1}^n\|_{l^\infty_{\tau, N}L^2} + \|\mathcal{F}_{2}^n\|_{l^\infty_{\tau, N}L^2} \leq C_{T} \left| \log \tau \right|^{ 2\over 3}  \tau^{2 \over 3}, \quad \ 
\|\mathcal{F}_{1}^n\|_{l^4_{\tau, N}L^4} + \|\mathcal{F}_{2}^n\|_{l^4_{\tau, N}L^4} \leq C_{T} \left| \log \tau \right|^{ 2\over 3}  \tau^{5 \over 12}.
\eeq
By using the same approach as before, we first prove by induction that for all $0 \leq  k \leq N_{1}$
\beq
\label{induc3} 
\|e^n\|_{X_{\tau,k}}:= { 1\over \left|\log \tau \right|^{ 2\over 3} \tau^{ 2 \over 3 }}\|e^n\|_{l^\infty_{\tau, k}L^2} + {1 \over  \tau^{ 19 \over 48}} \| e^n \|_{l^{4}_{\tau, k}L^4} \leq 8 C_{T}.
\eeq
Note that we propagate only \black the rate $\tau^{ 19 \over 48}$ \black for the $l^{4}_{\tau, N}L^4$ norm as we would expect $\tau^{5 \over 12} \left|\log \tau \right|^{2\over 3}$ from the estimate of the source term \eqref{erreur3d}. This is needed in order to close the argument below with this choice of norms. Moreover, as before this allows us to get by Sobolev embedding and \eqref{estfilter3} that
\beq
\label{enfilter3}
\| \varphi_{1}(-2 i \tau \Delta) e^n \|_{l^{\infty}_{\tau, k}L^4} \lesssim \| \varphi_{1}(-2 i \tau \Delta) e^n \|_{l^{\infty}_{\tau, k}H^{ 3 \over 4}}
\black\lesssim { 1 \over \tau^{3 \over 8}} \|e^n \|_{l^{\infty}_{\tau, k}L^2} \black 
\lesssim \tau^{7 \over 24} \left|\log \tau \right|^{2\over 3} \|e^n\|_{X_{\tau,k}}.
\eeq

From the same arguments as above, we get from \eqref{duhamelscheme2} and the discrete Strichartz estimates that
\beq
\label{13D}  
\|e^n\|_{l^\infty_{\tau,k+1}L^2} \leq  C_{T} \tau^{2 \over 3 } \left|\log \tau\right|^{ 2\over 3} + C \|G_{n}^1\|_{l^1_{\tau, k}L^2} + C \|G_{n}^2\|_{l^1_{\tau, k}L^2} + C { 1 \over \tau^{1 \over 8} } \|G_{n}^3 \|_{l^{8\over 5}_{\tau, k}L^{4\over 3} }.
\eeq
To estimate $\|G_{n}^1\|_{l^1_{\tau, k}L^2} $, we just use H\"older to get as before
$$
\|G_{n}^1\|_{l^1_{\tau, k}L^2} \leq C \| e^n\|_{l^\infty_{\tau,k}L^2} \left( \|u^K(t_{n})\|_{l^2_{\tau,k} L^\infty}^2 + \|\varphi_1(-2i \tau\Delta) \Pi_K u^K(t_{n})\|_{l^2_{\tau,k} L^\infty}^2\right).
$$
Next, the crucial observation is that since $\alpha = { 4 \over 3}$, we can use Remark \ref{remcond} to get that
\begin{multline*}
\|u^K(t_{n})\|_{l^2_{\tau,k} L^\infty} +\|\varphi_1(-2i \tau\Delta) \Pi_K u^K(t_{n})\|_{l^2_{\tau,k} L^\infty} \lesssim  T_{1}^{1\over 10}\left( \|u^K(t_{n})\|_{l^{5\over 2}_{\tau,k} W^{\sigma, {30 \over 7}}} \right.\\+ 
\left. \|\varphi_1(-2i \tau\Delta) \Pi_K u^K(t_{n})\|_{l^{5\over 2}_{\tau,k} W^{\sigma, {30\over 7}}}\right)
\end{multline*}
for $\sigma \in (21/30, \black 24/30)$. \black This allows us to use Proposition \ref{propuKinfty} and Corollary \ref{corfilter} to obtain that
\beq
\label{infty3D}
\|u^K(t_{n})\|_{l^2_{\tau,k} L^\infty} + \|\varphi_1(-2i \tau\Delta) \Pi_K u^K(t_{n})\|_{l^2_{\tau,k} L^\infty} \lesssim  C_{T}T_{1}^{1\over 10}
\eeq
and therefore
\beq
\label{Gn13D} 
\|G_{n}^1\|_{l^1_{\tau, k}L^2} \leq C_{T} T_{1}^{1\over 5} \| e^n\|_{l^\infty_{\tau,k}L^2}.
\eeq

For the estimate of $\|G_{n}^2\|_{l^1_{\tau, k}L^2}$, we can still write
\begin{multline*}
\|G_{n}^2\|_{l^1_{\tau, k}L^2} \leq C\left( \| e^n\|_{l^4_{\tau,k}L^4}^2  + \| e^n\|_{l^4_{\tau,k}L^4} \|  \varphi_{1}(-2i\tau \Delta) e^n\|_{l^4_{\tau,k}L^4}\right) \\ 
\cdot \left( \|u^K(t_{n})\|_{l^2_{\tau,k} L^\infty} +\|\varphi_1(-2i \tau\Delta) \Pi_K u^K(t_{n})\|_{l^2_{\tau,k} L^\infty} \right).
\end{multline*}
Consequently, by using again \eqref{infty3D}, we obtain that
\beq
\label{Gn23D} 
\|G_{n}^2\|_{l^1_{\tau, k}L^2} \leq C_{T} \left(\| e^n\|_{l^4_{\tau,k}L^4}^2 + \| e^n\|_{l^4_{\tau,k}L^4} \|  \varphi_{1}(-2i\tau \Delta) e^n\|_{l^4_{\tau,k}L^4}\right) \leq  C_{T} (\tau^{19 \over 24} + \tau^{33 \over 48}\left|\log \tau \right|^{ 2\over 3})\|e^n\|_{X_{\tau,k}}^2,
\eeq
where \black we have used \eqref{enfilter3}  and the fact that
$$ \| \varphi_{1}(-2i\tau \Delta) e^n\|_{l^4_{\tau,k}L^4} \leq  T   \| \varphi_{1}(-2i\tau \Delta) e^n\|_{l^\infty_{\tau,k}L^4}$$
to get the last estimate. \black 

It remains to estimate $\|G_{n}^3 \|_{l^{8\over 5}_{\tau, k}L^{4\over 3} }$. From H\"older's inequality, we get
$$ 
\|G_{n}^3 \|_{l^{8\over 5}_{\tau, k}L^{4\over 3} } \leq C \left(  \| e^{n}\|_{l^{24\over 5}_{\tau,k} L^4}^3 + \| e^{n}\|_{l^{16\over 5}_{\tau,k} L^4}^2  \| \varphi_{1}(-2i\tau \Delta) e^n\|_{l^\infty_{\tau,k}L^4}\right).
$$
By using the \black reverse inclusion rule \black 
for the discrete $l^p_{\tau}$ spaces,
$$ \|f\|_{l^p_{\tau} X} \lesssim  { 1 \over \tau^{{ 1 \over q} - {1 \over p }}}  \|f\|_{l^q_{\tau} X}, \quad p>q, $$
 we get
$$  
\| e^{n}\|_{l^{24\over 5}_{\tau,k} L^4}^3 \leq \left( {1 \over \tau^{1 \over 24}} \|e^{n}\|_{l^4_{\tau,k}L^4}\right)^3 \leq { 1 \over \tau^{1 \over 8}} \|e^{n}\|_{l^4_{\tau,k}L^4}^3.
$$
This yields by using again \eqref{enfilter3}
\begin{multline}
\label{Gn33D}
\|G_{n}^3 \|_{l^{8\over 5}_{\tau, k}L^{4\over 3} } \leq  C_{T} \left( { 1 \over \tau^{1 \over 8}} \|e^{n}\|_{l^4_{\tau,k}L^4}^3 + \| e^{n}\|_{l^{4}_{\tau,k} L^4}^2  \| \varphi_{1}(-2i\tau \Delta) e^n\|_{l^\infty_{\tau,k}L^4}\right)\\  
\leq C C_{T}^4 ( \tau^{17\over 16} + \tau^{13 \over 12} \left | \log \tau \right|^{2\over 3} ) \|e^n\|_{X_{\tau, k}}^3.
\end{multline}
Consequently, we deduce from \eqref{13D} and \eqref{Gn13D}, \eqref{Gn23D}, \eqref{Gn33D} and by using the induction assumption that
\beq
\label{en3D}
{ \|e^n\|_{l^\infty_{\tau,k+1}L^2} \over \tau^{2 \over 3 } \left|\log \tau \right|^{ 2\over 3} } \leq  C_{T}  + C_{T} T_{1}^{1 \over 5}  {  \|e^n\|_{l^\infty_{\tau,k+1}L^2} \over \tau^{2 \over 3 } \left|\log \tau \right|^{2\over 3} }  +  C C_{T}^3   \tau^{1 \over 48} + C C_{T}^3 { 1 \over \left| \log  \tau \right|^{2\over 3}}  \tau^{13 \over 48}.
\eeq
In a similar way, we can estimate $ \| e^n \|_{l^{4}_{\tau, k+1}L^4}$.  By using as previously 
 that we have the frequency localization  $\Pi_{2K} e^n= e^n$  \black  \black 
and the discrete Strichartz estimates, we get that
$$ 
\| e^n \|_{l^{4}_{\tau, k+1}L^4} \leq  C_{T} \left| \log \tau \right|^{2\over 3} \tau^{ 5 \over 12 } + C\left({ 1 \over \tau^{2 \over 3}}\right)^{1 \over 4}  \left({ 1 \over \tau^{1 \over 6}}\right)^{1 \over 2} \left(\|G_{n}^1\|_{l^1_{\tau, k}L^2} +  \|G_{n}^2\|_{l^1_{\tau, k}L^2}  + \left({ 1 \over \tau^{ 1 \over 6} }\right)^{3 \over 4}\|G_{n}^3 \|_{l^{{8\over 5}}_{\tau, k}L^{4\over 3}}\right).
$$
The additional loss  $\left({ 1 \over \tau^{2 \over 3}}\right)^{1 \over 4}$ comes from the fact that we need to use first the estimate
$$ \left\| \tau \sum_{k=0}^{n-1}  S_{K}(t_{n-k}) G_{k} \right\|_{l^{4}_{\tau, k+1}L^4} \lesssim  \left({ 1 \over \tau^{2 \over 3}}\right)^{1 \over 4}  \left\| \tau \sum_{k=0}^{n-1}  S_{K}(t_{n-k}) G_{k} \right\|_{l^{4}_{\tau, k+1}L^3}
$$
before using the discrete Strichartz estimates since $(4,4)$ is not admissible in dimension $3$. By using again \eqref{Gn13D}, \eqref{Gn23D}, \eqref{Gn33D}, we therefore obtain that
$$  
{\| e^n \|_{l^{4}_{\tau, k+1}L^4} \over \tau^{19 \over 48} } \leq C_{T} \left|\log \tau\right|^{2\over 3} \tau^{1 \over 48}+ C_{T} T_{1}^{1\over 5} \|e^n\|_{X_{\tau,k}}+C C_{T}^3 \tau^{1\over 48} + C C_{T}^3  \tau^{7 \over 24}.
$$
By combining \eqref{en3D} and the last estimate, we obtain that
$$ 
\|e^n\|_{X_{\tau,k+1}} \leq 2 C_{T} +   C_{T} T_{1}^{1\over 5} \|e^n\|_{X_{\tau,k+1}} + C C_{T}^3 \tau^\delta
$$
for some $\delta >0$. Therefore, we can finish the proof as above.

\section{Proof of the discrete Strichartz estimates.}
\label{sectionproofstrichartz}

\subsection{Dispersive estimates}

Let us start with the proof of a dispersive inequality.

\begin{lemma}
\label{lemdisp}
There exists $C>0$ such that \black for every $K \geq 1$, every $p \in [2, \infty]$, every $t\in \mathbb R$, and every $f\in L^{p'}$, \black we have the estimate
$$ 
\|S_{K} (t)f\|_{L^p} \leq C { K ^{d ( 1 - { 2 \over p})} \over 1 + |t|^{{d \over 2}( 1 - {2 \over p})} }\|f\|_{L^{p'}}.
$$
\end{lemma}

\begin{proof}
In this proof $C>0$ will stand for a number \black independent of $K$. \black 
Let us observe that with the choice of $\Pi_{K}$ as in \eqref{PiKdef}, we can write
$$  
S_{K} (t)f = \rho_{\epsilon} *\left( e^{i t \Delta } ( \rho_{\epsilon} *f )\right)
$$
where $\rho_{\epsilon}= {1 \over \epsilon^d} \rho\left( {x \over \epsilon}\right)$, $ \epsilon = {1 \over K }$, $\rho(x)= \mathcal{F}^{-1}(\chi) (x) \in L^1$. From Young's inequality for convolutions and the standard dispersive estimate for $e^{it \Delta}$, we thus get that
$$ 
\|S_{K} (t) f\|_{L^\infty} \leq C  \|\rho_{\epsilon}\|_{L^1}  {1 \over |t|^{d \over 2}} \| \rho_{\epsilon}\|_{L^1}  \| f \|_{L^1}  \leq C {1 \over |t|^{d \over 2}}  \| f \|_{L^1}, \quad t \neq 0.
$$
For $|t| \leq 1$, we use the estimate
$$ 
\|S_{K} (t) f\|_{L^\infty} \leq \| \rho_{\epsilon}\|_{L^2} \| e^{it \Delta } (\rho_{\epsilon} *f) \|_{L^2} \leq   \| \rho_{\epsilon}\|_{L^2} \|\rho_{\epsilon}*f\|_{L^2} \leq   \| \rho_{\epsilon}\|_{L^2}^2 \|f\|_{L^1} \leq C K^d  \| f \|_{L^1}.
$$
By combining the two inequalities, we get that
$$  
\|S_{K} (t)f\|_{L^\infty} \leq C {K ^d \over (1 + |t|^{d\over 2})} \|f\|_{L^1}.
$$
Since we also have that
$$ 
\|S_{K}(t) f \|_{L^2} \leq \|f\|_{L^2}, 
$$
we get the desired estimate by complex interpolation.
\end{proof}

\subsection{Proof of Theorem \ref{theoDSE}}
\label{proofDSE}
   
By a scaling argument, \black it is sufficient to \black study the case $\tau= 1$. Indeed, we have that
\beq
\label{scaleS} 
S_{K} (t) \phi(x)=  \left(S_{K\tau^{1\over 2}} \left( {t \over \tau}\right) \phi( \tau^{ 1\over 2} \cdot)\right) \left( {x \over \tau^{1\over 2} } \right).
\eeq
Therefore, it suffices to prove the estimates
\begin{align}
\label{unS}& \left\| S_{K \tau^{1\over 2} }(n) f \right\|_{l^pL^q} \leq C  ( K \tau^{1\over 2 })^{2\over p }   \|f \|_{L^2}, \\
\label{deuxS} & \left\| \sum_{n \in \mathbb{Z} }  S_{K \tau^{1 \over 2}} (-n) F_{n} \right\|_{L^2} \leq C  ( K \tau^{1\over 2 })^{2\over p}  \| F\|_{l^{p'} L^{q'} }, \\
\label{troisS}  &\left\| \sum_{k \in \mathbb{Z}}   S_{K \tau^{1 \over 2}} (n-k+ s)  F_{k} \right\|_{l^{p_{1}}L^{q_{1}} } \leq C ( K \tau^{1\over 2 } )^{ {2\over p_{1}} + { 2 \over p_{2} }   }  \| F\|_{l^{p_{2}'} L^{q_{2}'} },
\end{align}
where $l^p$ now stands for the usual discrete norms on sequences ($\| u\|_{l^p X}= (\sum_{n\in \mathbb{Z}} \|u_{n}\|_{X}^p)^{1 \over p}$). These estimates are equivalent through the usual \black $\mathcal T\mathcal T^*$ \black argument. If we define $(\mathcal T f)_{n}=  S_{K \tau^{1\over 2} }(n) f$. Then
\begin{align*}
&  \mathcal T^* F= \sum_{k \in \mathbb{Z} }  S_{K \tau^{1 \over 2}} (-k) F_{k},\\
& (\mathcal T\mathcal T^* F)_{n}= \sum_{k \in \mathbb{Z} }  S_{K \tau^{1 \over 2}} (n-k) F_{k}
\end{align*}
and
$$ 
\|\mathcal T\|_{L^2\rightarrow l^pL^q}= \|\mathcal T^*\|_{l^{p'}L^{q'}\rightarrow L^2} = \|\mathcal T \mathcal T^* \|_{l^{p'}L^{q'}\rightarrow  l^pL^q }^{1 \over 2}.
$$
Note that the estimate \eqref{troisS} corresponds to an estimate of $ \mathcal T e^{is \Delta } \mathcal T^*$ so that the estimate of $\mathcal T\mathcal T^*$ is a special case with $s=0$. We shall  first prove the estimate for $\mathcal T e^{is \Delta } \mathcal T^*$. We write that uniformly for $s \in [-8, 8]$,
$$  
\|\mathcal T e^{is \Delta }\mathcal T^* F \|_{l^p L^q } \leq \left\|  \sum_{k \in \mathbb{Z}} \| S_{K \tau^{1 \over 2}} (n-k+ s)  F_{k}\|_{L^q} \right\|_{l^p} \leq   C \left\|  \sum_{k \in \mathbb{Z}}  { (K \tau^{1\over 2} )^{d ( 1 - { 2 \over q})} \over 1 + |n-k|^{{d \over 2}( 1 - {2 \over q})}  }    \|F_{k}\|_{L^{q'}} \right\|_{l^p},
$$
where the last inequality comes from Lemma \ref{lemdisp} applied to $S_{K \tau^{1\over 2}}$ for $K \tau^{1 \over 2} \geq 1$. From a discrete version of the \black Hardy--Littlewood--Sobolev inequality (see again \cite{Ignat11}), we then obtain that \black 
$$  
\|\mathcal T e^{is \Delta} \mathcal T^* \|_{l^{p'}L^{q'}\rightarrow  l^pL^q }^{1 \over 2} \leq C (K \tau^{1\over 2} )^{{d\over 2} ( 1 - { 2 \over q})} =  C (K \tau^{1\over 2} )^{2 \over p} 
$$
by using the admissibility relation as long as $p>2$. This yields \eqref{unS} and \eqref{deuxS}. To get the general form of \eqref{troisS}, it suffices to estimate $\mathcal T e^{is \Delta}\mathcal T^*$ by composing the estimate for $\mathcal T$, the $L^2$ continuity of $e^{is \Delta}$ and the estimate for $\mathcal T^*$. Once we have \eqref{troisS}, the truncated version comes from the discrete Christ--Kiselev lemma as in \cite{IZ09} except in the case that  $(p_{1}, q_{1})$ and $(p_{2}, q_{2})$ are the endpoint, but we excluded it for these estimates. One could also use a classical interpolation argument.

\subsection{Proof of Corollary \ref{corDSE}}\label{sec:corDSE}

We shall use the Littlewood--Paley decomposition in order to convert the loss in the estimates of Theorem \ref{theoDSE} into a loss of derivative. Let us recall some basic facts, we refer to the book \cite{Bahouri-Chemin-Danchin} for the proofs. We take a partition of unity of the form
$$ 
1= \varphi_{-1}(\xi)+ \sum_{k \geq 0} \varphi_{k}(\xi)
$$
where $\varphi_{-1}$ is supported in the ball $\overline{B}(0,1)$ and each $\varphi_{k}(\xi)= \varphi(\xi/2^k)$, $k\geq 0$ is supported in the annulus $2^{k-1} \lesssim |\xi | \lesssim 2^{k+1}$. We can then decompose any tempered distribution as
$$ 
u = \sum_{k \geq -1} u_{k}, \qquad  \mathcal{F}(u_{k}) (\xi)= \varphi_{k}(\xi) \hat u(\xi).
$$
We shall only use the following facts:

\begin{itemize}
\item {\bf Bernstein inequality.} 
For every $\sigma \geq 0$ and every $p \in [1,\infty]$, there exist constants $c>0$ and $C>0$ such that for every $k \geq 0$, we have
\beq
\label{bernstein}         
c2^{\sigma k}\| ( \varphi_{k}(-i\nabla)) u\|_{L^p} \leq \| \left|-i\nabla\right|^\sigma ( \varphi_{k}(-i\nabla) u )\|_{L^p} \leq C 2^{\sigma k} \| ( \varphi_{k}(-i\nabla) u)\|_{L^p}.
\eeq
\item {\bf Characterization of $L^q$ spaces.}
For $q \geq 2$, the $L^q$ norm of a function is equivalent to the norm 
\beq
\label{Lqlittle} 
\left\|  \left( \sum_{k \geq -1} |u_{k}|^2 \right)^{1 \over 2} \right\|_{L^q}:= \|(u_{k}) \|_{L^q l^2}.
\eeq
Note that when $q=2$, we can invert the order of summation so that
$$ 
\|u\|_{L^2} \sim  \left(\sum_{k\in\mathbb Z } \|u_{k}\|_{L^2}^2\right)^{1 \over 2}= \|(u_{k})\|_{l^2 L^2},
$$
where $\sim$ denotes the equivalence of norms. Further, by Minkowski's inequality, we have that
$$ 
\|u\|_{L^q} \lesssim \| (u_{k})\|_{l^2L^q}.
$$
\end{itemize}

Let us first prove \eqref{Tloss}. By using the Littlewood--Paley decomposition, \black we first note that \black thanks to Minkowski's inequality, we have
\beq
\label{fin1} 
\| S_{K}(n\tau) u\|_{l^p_{\tau}L^q} \lesssim \left\| \left( \sum_{k \geq -1} \|S_{K}(n \tau) u_{k}\|^2_{L^q} \right)^{1 \over 2} \right\|_{l^p_{\tau}} \lesssim \left( \sum_{k \geq -1} \|S_{K}(n \tau) u_{k}\|^2_{l^p_{\tau}L^q} \right)^{1 \over 2}
\eeq
since $p \geq 2$. To estimate the terms inside the sum, we observe that
$$ 
S_{K}(n \tau) u_{k} = S_{2^k} (n \tau)\Pi_{K}u_{k}.
$$
Note that, because of the truncation $\Pi_{K}$, the sum is actually finite. We sum only over the $k$ such that $2^k \lesssim  K= \tau^{-{\alpha \over 2}}$.

If $2^k \tau^{1 \over 2 } \lesssim 1$, we can also write
$$  
S_{K}(n \tau) u_{k} = S_{ \tau^{ -{1\over 2} }} (n\tau) \Pi_{K}u_{k}.
$$
Therefore, by Theorem \ref{theoDSE}, we obtain the  estimate without loss
$$  
\|S_{K}(n \tau) u_{k}\|_{l^p_{\tau}L^q} \leq C \|u_{k}\|_{L^2}.
$$
If $\tau^{-{1 \over 2}}  \leq 2^k \leq \tau^{-{\alpha \over 2}}$, we obtain that
$$  
\|S_{K}(n \tau) u_{k}\|_{l^p_{\tau}L^q} \leq C (2^k \tau^{1\over 2})^{2\over p} \|u_{k}\|_{L^2}.
$$
Consequently, from the two sides of the Bernstein inequality, we obtain
$$   
\|S_{K}(n \tau) u_{k}\|_{l^p_{\tau}L^q} \leq  C (\tau^{1 \over 2})^{2\over p} \|u_{k}\|_{H^{ 2\over p}} \leq C \|u_{k}\|_{H^{ {2\over p}(1 - { 1\over  \alpha}) }}.
$$
This yields thanks to \eqref{fin1}
$$ 
\| S_{K}(n\tau) u\|_{l^p_{\tau}L^q} \lesssim \left(\sum_{k \geq -1} \| u_{k}\|_{H^{ {2\over p}(1 - { 1\over  \alpha}) }}^2 \right)^{1\over 2} \lesssim \|u\|_{H^{ {2\over p}(1 - { 1\over  \alpha}) }},
$$
which gives \eqref{Tloss}.

The proof of \eqref{TT*loss} follows exactly the same lines.

\section{Some technical estimates}
\subsection{Properties of the filter function}
\begin{lemma}
\label{lemfilter1}
 We have the following properties:
 \begin{itemize}
 \item  For every $p \in [1, \infty],$  there exists $C>0$ such that for every $\tau  \in (0, 1]$, 
 \beq
  \label{estfilter1}
  \left\| \varphi_{1}(-2 i \tau \Delta) \Pi_{\tau^{-{1 \over 2  } }} f \right\|_{L^p}  \leq C \| f\|_{L^p} \quad \mbox{for all } f \in L^p .
  \eeq
  \item For every $p \in (1, \infty),$  there exists $C>0$ such that for every $\tau  \in (0, 1]$
  \beq
  \label{estfilter2} \left\| { 1 - \Pi_{\tau^{-{1 \over 2  } } }\over  2 i \tau \Delta } f \right\|_{L^p}  \leq C \| f\|_{L^p} \quad \mbox{for all } f \in L^p.
    \eeq
  \item For every $s \in [0, 2 ]$, here exists $C>0$ such that for every $\tau  \in (0, 1]$
  \beq
  \label{estfilter3}
   \left\| \varphi_{1}(-2 i \tau \Delta) f \right\|_{H^s} \leq {C \over \tau^{s \over 2}} \|f\|_{L^2} \quad \mbox{for all } f \in L^2.
   \eeq
  \end{itemize}

\end{lemma}

\begin{proof}
 We first prove \eqref{estfilter1}.
 Let us  set  by $L_{\tau} = \varphi_{1}(-2 i \tau \Delta) \Pi_{\tau^{-{1 \over 2  } }}.$ We first observe that
 $$L_{\tau } f =  \left( L_{1} (f (\tau^{1 \over 2 } \cdot)\right) \left({ \cdot \over \tau^{1\over 2} } \right).$$
 Therefore, by scaling, it suffices to prove the estimate  \eqref{estfilter1} for $L_{1}$.
  Next, we can also write that
  $$ L_{1}f= \Phi * f$$
  where $\Phi = \mathcal{F}^{-1} m_{1}$  with 
  $m_{1}(\xi)=  \varphi_{1}(2 i |\xi|^2) \chi^2( \xi).$
   Since $\chi$ is compactly supported and $\varphi_{1}$ is smooth, we have that $m_{1}$ and therefore $\Phi$ are in the Schwartz class, 
     therefore  we get in particular that  $\Phi\in L^1$ and the result follows from standard properties of convolutions.
     
     By the same scaling argument, to prove \eqref{estfilter2}, it suffices to prove the estimate with $\tau = 1$.
      We observe again that this amounts to prove the $L^p$ continuity of the Fourier multiplier  by 
      $ m_{2}(\xi)= { 1 - \chi^2(\xi) \over - 2 i |\xi|^2}.$  We observe that $m_{2}$ is a smooth bounded function that satisfies in addition
       the estimate  
       $$ | \partial^\alpha m_{2}(\xi) | \leq { C_{\alpha}  \over | \xi |^{\alpha}} \quad \text{for all } \xi \in \mathbb{R}^d$$
       for every $\alpha \in \mathbb{N}^d$. Consequently, the result follows from the H\"ormander--Mikhlin Theorem.
       
     To get \eqref{estfilter3}, it suffices to observe that the function
     $ \varphi_{1}(2i \tau  |\xi|^2 )  (1 + \tau  | \xi|^2)^{s \over 2}$ is uniformly bounded by a constant independent of $\tau$ 
      and to use the Bessel identity.
      \end{proof}
  \subsection{A localized critical Sobolev embedding}
   We have the following classical borderline Sobolev estimate for frequency localized functions in dimension $3$.
  \begin{lemma}
  \label{lemsobfin}
   The exists $C>0$ such that for every $u \in W^{1, 3}(\mathbb{R}^3)$ with $\mbox{Supp } \hat u \subset B(0, 4K)$, $K \geq 1$, 
   we have
   $$ \|u \|_{L^\infty} \leq C (\log K)^{2 \over 3} \|u\|_{W^{1,3}}.$$
  \end{lemma}
\begin{proof}
 By using the Littlewood--Paley decomposition introduced in the previous section and the triangular  inequality, we have that
 $$ \|u \|_{L^\infty} \leq \sum_{ 2^k \leq 4K} \|u_{k}\|_{L^\infty}.$$
 Note that the sum is finite thanks to the assumption on the support of the Fourier transform of $u$.
  Next, since  $u_{k}=  \Pi_{ 4\cdot 2^k} u$, we get from Young's
   inequality for convolutions that
   $$   \|u_{k}\|_{L^\infty} \lesssim 2^k \|u^k \|_{L^3}.$$
    Therefore, by using the Bernstein inequality \eqref{bernstein}, we get that
   $$ \|u \|_{L^\infty} \lesssim \sum_{ 2^k \leq 4K} 2^k \|u_{k}\|_{L^3} \lesssim \|u\|_{L^3} +  \sum_{ 1 \leq 2^k \leq 4K}  \| \nabla u_{k} \|_{L^3}.$$
    Next, from H\"older's inequality and Fubini we get that
    $$  \sum_{ 1 \leq 2^k \leq 4K}  \| \nabla u_{k} \|_{L^3} \lesssim (\log K)^{ 2 \over 3}  \left( \sum_{ k \geq -1}  \| \nabla u_{k} \|_{L^3}^3\right)^{ 1 \over 3} \lesssim (\log K)^{ 2 \over 3} \left\| \left( \sum_{k \geq -1} | \nabla u_{k}|^3 \right)^{ 1 \over 3} \right\|_{L^3}.$$
     Since,  by the embedding of discrete $l^p$ spaces, we have that 
     $$   \left( \sum_{k \geq -1} | \nabla u_{k}|^3 \right)^{ 1 \over 3}  \lesssim  \left( \sum_{k \geq -1} | \nabla u_{k}|^2 \right)^{ 1 \over 2}, $$
    we finally obtain that 
    $$  \|u \|_{L^\infty} \lesssim  \|u\|_{L^3}+  (\log K)^{ 2 \over 3} \left\|\left( \sum_{k \geq -1} | \nabla u_{k}|^2 \right)^{ 1 \over 2} \right\|_{L^3}
     \lesssim (\log K)^{ 2 \over 3} \|u\|_{W^{1, 3}},$$
     where the final estimate comes from \eqref{Lqlittle}. 
  
\end{proof}

\end{document}